\documentclass[a4paper,10pt,fleqn]{amsart}

\usepackage[utf8]{inputenc}
\usepackage[T1]{fontenc}
\usepackage{amsmath}
\usepackage{amsfonts}
\usepackage{amsthm}
\usepackage{amssymb}
\usepackage{bbm}
\usepackage[colorlinks,breaklinks=true]{hyperref}
\usepackage{appendix}
\usepackage{aliascnt}
\usepackage{marginnote}

\usepackage{tikz,pgfplots,enumerate}
\pgfplotsset{compat=1.11}
\usetikzlibrary{calc,angles,quotes}
\tikzset{help lines/.style=very thin}
\tikzset{Grid/.style={help lines,color=gray!20}}

\makeatletter

\AtBeginDocument{%
  \@ifpackageloaded{refcheck}{%
    \@ifundefined{hyperref}{}{%
      \let\T@ref@orig\T@ref
      \def\T@ref#1{\T@ref@orig{#1}\wrtusdrf{#1}}%
      \let\@refstar@orig\@refstar
      \def\@refstar#1{\@refstar@orig{#1}\wrtusdrf{#1}}%
      \DeclareRobustCommand\ref{\@ifstar\@refstar\T@ref}%
    }%
  }{}%
}

\makeatother

\title[On fully nonlinear mean field games]%
{On fully nonlinear parabolic mean field games with 
nonlocal and local diffusions}

\author{Indranil Chowdhury}
\address{\parbox{.8\linewidth}{{\normalfont\textbf{I.~Chowdhury}}\medskip\\
Indian Institute of Technology - Kanpur,\\
Department of Mathematics and Statistics, \\
Kalyanpur, Kanpur - 208016, India\medskip}}
\email{indranil@iitk.ac.in}

\author{Espen R.~Jakobsen}
\address{\parbox{.8\linewidth}{{\normalfont\textbf{E.~R.~Jakobsen}}\medskip\\
Institutt for matematiske fag, NTNU,\\
7491 Trondheim, Norway\medskip}}
\email{espen.jakobsen@ntnu.no}

\author{Mi\l{}osz Krupski}
\address{\parbox{.8\linewidth}{{\normalfont\textbf{M.~Krupski}}\medskip\\
Instytut Matematyczny, Uniwersytet Wroc\l{}awski,\\
pl.~Grunwaldzki 2/4, 50-384 Wroc\l{}aw, Poland\medskip\\
Prirodoslovno--matemati\v{c}ki fakultet, Sveu\v{c}ili\v{s}te u Zagrebu,\\
Horvatovac 102a, 10000 Zagreb, Croatia\medskip}}
\email{milosz.krupski@uwr.edu.pl}

\makeatletter
\@namedef{subjclassname@2020}{%
  \textup{2020} Mathematics Subject Classification}
\makeatother

\date{\today}

\let\OLDenum\enumerate
\renewcommand\enumerate{\vspace{0.25\baselineskip}\OLDenum\setlength{\itemsep}{0.5\baselineskip}}
\newenvironment{description*}%
  {\vspace{0.25\baselineskip}\begin{description}
    \setlength{\itemsep}{0.5\baselineskip}
    \setlength{\mathindent}{1.5\leftmargin}
  }
  {\end{description}}
   
\makeatletter
\let\orgdescriptionlabel\descriptionlabel
\renewcommand*{\descriptionlabel}[1]{%
 \let\orglabel\label
 \let\label\@gobble
 \phantomsection
 \edef\@currentlabel{#1}%
 \let\label\orglabel
 \orgdescriptionlabel{#1}%
}
\makeatother

\newtheorem{theorem}{Theorem}[section]
\newaliascnt{lemma}{theorem}\newtheorem{lemma}[lemma]{Lemma}\aliascntresetthe{lemma}
\newaliascnt{corollary}{theorem}\newtheorem{corollary}[corollary]{Corollary}\aliascntresetthe{corollary}
\newaliascnt{proposition}{theorem}\newtheorem{proposition}[proposition]{Proposition}\aliascntresetthe{proposition}
\newaliascnt{conjecture}{theorem}\aliascntresetthe{conjecture}
\theoremstyle{remark}
\newaliascnt{remark}{theorem}\newtheorem{remark}[remark]{Remark}\aliascntresetthe{remark}
\theoremstyle{definition}
\newaliascnt{definition}{theorem}\newtheorem{definition}[definition]{Definition}\aliascntresetthe{definition}

\def\equationautorefname~#1\null{problem~\upshape{(#1)}\null}
\def\partnautorefname~#1\null{Part\,(#1)\null}
\def\itemautorefname~#1\null{\,(#1)\null}

\newcounter{step}[theorem]
\newcounter{partn}[theorem]
\renewcommand{\thepartn}{\textit{\roman{partn}}}

\newcommand{\bulletdiam}{\raisebox{\dimexpr.5\fontcharht\font`S-.5\height}{$\diamond$}\ }
\newenvironment{proof*}{\begin{proof}\setcounter{step}{0}}{\end{proof}}
\newenvironment{step}{\refstepcounter{step}\bulletdiam
{\textit{Step \thestep}.}}
{\smallskip\par}
\newenvironment{step*}[1]{\refstepcounter{step}\bulletdiam {\textit{Step \thestep. #1}.}}
{\smallskip\par}
\newenvironment{part*}
{\refstepcounter{partn}\bulletdiam{\textit{Part {\rm (}\thepartn\,{\rm )}}.}}
{\smallskip\par}

\newcounter{parn}[remark]
\renewcommand{\theparn}{\textit{\alph{parn}}}
\newcommand*{\npar}{\refstepcounter{parn}\ifnum\value{parn}=1{\par({\theparn})\ }\else{\smallskip\par ({\theparn})\ }\fi}
\def\parnautorefname~#1\null{\,({#1})\null}

\DeclareMathOperator{\supp}{supp}
\DeclareMathOperator{\sgn}{sgn}
\DeclareMathOperator{\tr}{tr}
\DeclareMathOperator{\loc}{loc\,}
\DeclareMathOperator{\nloc}{nloc\,}

\DeclareMathOperator*\limp{\vphantom{p}lim}
\DeclareMathOperator{\Dom}{Dom}
\newcommand{\dt} {\partial_t}
\newcommand{\KHJ}{K_{\text{\textit{HJB}}}}
\newcommand{\FPm}{\mathcal{M}}
\newcommand{\R} {\mathbb{R}}
\newcommand{\N} {\mathbb{N}}
\newcommand{\X} {{\R^d}}
\newcommand{\fL}{\mathcal{L}}

\newcommand{\fLs}{\fL^*}
\newcommand{\PX}{\mathcal{P}(\X)}
\newcommand{\Mb}{\mathcal{M}_b(\X)}
\newcommand{\T}{\mathcal{T}}
\newcommand{\Tb}{\overline{\T}}
\newcommand{\rtc}{\theta}

\newcommand{\hd}{\alpha}

\newcommand{\hb}{\beta} 
\newcommand{\Holder}[1]{\mathcal{C}^{#1}}
\newcommand{\Hb}[1]{\Holder{#1}_b(\X)}
\newcommand{\HH}[3]{\mathcal{C}^{#1,#2}([0,t]\times B_{#3}(x))}
\newcommand{\HHb}[2]{\mathcal{C}^{#1,#2}_b(\T\times\X)}
\newcommand{\CPX}{C(\Tb,\PX)}
\newcommand{\CaPX}[1]{\Holder{#1}(\Tb,\PX)}
\newcommand{\CLX}{C(\Tb,L^1(\X))}
\newcommand{\LC}[1]{B(\T, \Hb{#1})}
\newcommand{\CC}[1]{C_b(\T, \Hb{#1})}

\newcommand{\Cb}{C_b(\T\times\X)}
\newcommand{\Cbb}{C_b(\Tb\times\X)}
\newcommand{\Cu}{C(\T\times\X)}
\newcommand{\D}{C_c^\infty(\Tb\times\X)}
\newcommand{\BUC}{\text{\textit{BUC}}(\X)}
\newcommand{\UC}{\text{\textit{UC}}}
\newcommand{\cF}{\mathfrak{f}}
\newcommand{\cG}{\mathfrak{g}}
\newcommand{\cD}{\mathcal{R}}
\newcommand{\cB}{\mathcal{B}}
\newcommand{\HJcD}{\mathcal{S}_{\text{\textit{HJB}}}}
\newcommand{\uT}{g}
\newcommand{\unT}[1]{g_{#1}}
\renewcommand{\epsilon}{\varepsilon}
\newcommand{\fLlow}{\fL_r}
\newcommand{\fLhigh}{\fL^r}

\makeatletter
\def\@tocline#1#2#3#4#5#6#7{\relax
  \ifnum #1>\c@tocdepth 
  \else
    \par \addpenalty\@secpenalty\addvspace{#2}%
    \begingroup \hyphenpenalty\@M
    \@ifempty{#4}{%
      \@tempdima\csname r@tocindent\number#1\endcsname\relax
    }{%
      \@tempdima#4\relax
    }%
    \parindent\z@ \leftskip#3\relax \advance\leftskip\@tempdima\relax
    \rightskip\@pnumwidth plus4em \parfillskip-\@pnumwidth
    #5\leavevmode\hskip-\@tempdima
      \ifcase #1
       \or\or \hskip 1em \or \hskip 2em \else \hskip 3em \fi%
      #6\nobreak\relax
    \hfill\hbox to\@pnumwidth{\@tocpagenum{#7}}\par
    \nobreak
    \endgroup
  \fi}
\makeatother

\begin{document}
\begin{abstract}
 We introduce a class of fully nonlinear mean field games posed in $[0,T]\times\X$.
 We justify that they are related to controlled local or nonlocal diffusions,
 and more generally in our setting, to a new control interpretation involving
 time change rates of stochastic (L\'evy) processes.
 The main results are existence and uniqueness of solutions under general assumptions.
 These results are applied to
 non-degenerate equations --- including both local second order and nonlocal with fractional Laplacians. 
 Uniqueness holds under  monotonicity of couplings and convexity of the Hamiltonian, but neither monotonicity nor convexity need to be strict.
 We consider a rich class of nonlocal operators and
 processes and develop tools to work in the whole space without explicit moment assumptions.
\end{abstract}

\maketitle

 \begin{section}{Introduction}
 In this paper we introduce a new model of \emph{mean field games} and analyse it using PDE methods.
  Mean field games are limits of $N$-player stochastic games as $N\to\infty$, under certain assumptions allowing for the mean field limit to exist. 
 The Nash equilibria are characterized by a coupled system of PDEs called the \emph{mean field game system}, where the value function of the generic player is given by a backward  Hamilton--Jacobi--Bellman equation and the distribution of players by a forward Fokker--Planck equation. 
 The mathematical theory of such problems was introduced by Lasry--Lions~\cite{MR2269875,MR2271747,MR2295621}
 and Huang--Caines--Malham\'e~\cite{MR2346927,MR2344101} in 2006,
 and today this is a large and rapidly expanding field.
 This research is mostly focused on either PDE or stochastic approaches.
 Extensive background and recent developments can be found in e.g.~\cite{MR4214773,MR3134900,MR3752669,MR3753660,MR3559742,MR3967062,
 MR2762362}
 and the references therein.

 In contrast to the more classical setting,
 we allow not only the drift of a stochastic process to be controlled but also the diffusion.
 To be more precise, the players control the time change rate of a L\'evy process. 
 If the (diffusion) process is self-similar like a Brownian motion or an $\alpha$-stable process,  this is equivalent to a classical controlled
 diffusion \cite{MR2179357,MR2322248} (see \autoref{section:stochastics} and \cite{IJK03b} for more details). 
In our setup the backward equation is fully nonlinear, and the system may be strongly degenerate and local or nonlocal. 
  Problems sharing some of these features have been addressed before.
 In~\cite{MR3399179} the authors allow for a degenerate diffusion,
 but it is not controlled and there are restrictions on its regularity, cf.~\cite{MR2607035,MR3391701}.
 There are recent results on mean field games with nonlocal (uncontrolled) diffusion
 involving L\'evy operators~\cite{MR3912635,MR3934106,MR4309434,MR4223351}.
 See also~\cite{MR3992040} for a problem involving fractional time derivatives.

  Control problems/games have many applications throughout the sciences, engineering, and economics.
Controlled diffusions \cite{MR2179357} appear e.g.~in portfolio optimization in finance, cf.~\cite{MR2178045, MR2380957, MR2533355}. Despite the many applications in economics \cite{MR2762362,MR3363751}, see also \cite{MR0172689}, control of the diffusion is a rare and novel subject in the context of mean field games. So far it has been addressed mostly by stochastic methods:  \cite{MR3332857} introduces an approach based on relaxed controls and martingale problems to show existence (without uniqueness) of probabilistic solutions to very general local mean field games, see further developments in e.g.~\cite{BT22}. Mean field games of controls are considered in \cite{djete2023mean}, and~\cite{MR4158808,MR3980873} consider extensions to problems perturbed by bounded nonlocal operators.
Some results by PDE methods can be found in \cite{Ricciardi}, as well as \cite{MR4361908, MR4702626}
for uniformly elliptic (stationary second order) problems. Except for \cite{Ricciardi}, there
seem to be no prior uniqueness results for fully nonlinear problems by any methods.

We focus on the case where the generic player has a single control, which addresses most of the novelties. We also explain how to include a separately controlled drift and hence the corresponding first-order terms in the PDEs.
To be precise, we mainly study derivation, existence, and uniqueness questions for the mean field game system 
 \begin{align}\label{eq:mfg}
  \left\{ 
  \begin{aligned} 
   -\dt u &= F(\fL u) +\cF(m)\quad&\text{on $\T\times\X$},\\
   &u(T)=\cG(m(T))\quad&\text{on $\X$},\\[0.2cm]
   \dt m & = \fLs(F'(\fL u)\,m)\quad&\text{on $\T\times\X$},\\
   &m(0)=m_0\quad&\text{on $\X$},
  \end{aligned}
  \right.
 \end{align}
 where $\T=(0,T)$ for a fixed $T>0$.
 We assume $\fL$ to be a L\'evy operator with triplet $(c,a,\nu)$,
 an infinitesimal generator of a L\'evy process (see \cite[\S2.1]{MR3156646}).
 The (formal) adjoint $\fLs$  of~$\fL$ is also a L\'evy operator. Typical examples are the Laplacian $\Delta$, $(c,a,\nu)=(0,I,0)$, the fractional Laplacian $-(-\Delta)^\sigma$, $(c,a,\nu)=(0,0,\bar c\,{|z|^{-d-2\sigma}}\,dz)$ for $\bar c>0$ and $\sigma\in(0,1)$, and tempered, nonsymmetric, and even degenerate elliptic operators. 
 We discuss more in \autoref{sec:enadu}.

 A semi-rigorous derivation of \autoref{eq:mfg} is given in \autoref{section:stochastics}, starting with a precise interpretation of the control problem for the generic player in terms of the time change rate of the L\'evy process.
 Our derivation leads to a Hamiltonian $F$ which is convex and non-decreasing, 
 an optimal feedback control $\theta^*=F'(\fL u)$, and 
 ultimately to the mean field game system \eqref{eq:mfg} which is then parabolic.
 It is coupled through the running and terminal costs $\cF$ and $\cG$ and the optimal feedback control. In this paper $\cF$ and $\cG$ are smoothing (nonlocal) couplings. 
 
 Our first objective is to study the well-posedness of \autoref{eq:mfg} with a (nearly) minimal 
 set of assumptions  \ref{a:F1}--\ref{a:fg2}, naturally arising from the analysis in \autoref{section:stochastics}. This is matched with one of the weakest solution concepts where the feedback control is well-defined: classical solutions $u$ and measure-valued distributional solutions $m$.
Reworking the mean field games arguments to fit our setting,  
we reduce the question of well-posedness to a set of general conditions \ref{R1}--\ref{U2} describing the properties of solutions of the uncoupled equations making up \autoref{eq:mfg}: solvability, stability, and regularity of the Hamilton--Jacobi--Bellman equation \eqref{eq:hjb}, and uniqueness of the Fokker--Planck equation \eqref{eq:fp} --- see below. 

To prove uniqueness, we impose monotonicity assumptions on the couplings. Improving on previous results, we need neither strict convexity of the Hamiltonian nor strict monotonicity of the couplings. 
Existence for \autoref{eq:mfg} holds under much weaker assumptions than uniqueness, in part because we need no uniqueness for 
the Fokker--Planck equation. 
We exploit the Kakutani--Glicksberg--Fan fixed point theorem (a generalization of the Schauder theorem), relying on a stability result for sets of solutions of the Fokker--Planck equation. It is based on new tightness arguments which require no moment assumptions in $\X$ on $m$ or $\nu$.\footnote{In the mean field game literature, $m$ is usually continuous in the Wasserstein $d_1$ distance and has two bounded moments, whereas here we work with the Rubinstein--Kantorovich (or bounded--Lipschitz) distance $d_0$ and no explicit moment bounds. The challenge is to preserve compactness.}
This approach is of independent interest and has already been exploited in  \cite{MR4309434,chowdhury2021numerical,jakobsen2023master}.

The second objective is to verify  the abstract conditions \ref{R1}--\ref{U2} and hence obtain well-posedness
in concrete cases. We consider
two non-degenerate problems --- local (\autoref{sec:ndeg-loc}) and nonlocal (\autoref{sec:ndeg-nloc}). In an upcoming paper \cite{IJK03b}, we also show that our findings can be applied to certain strongly degenerate problems. In \autoref{sec:extension} we formulate the results for a combination of controlled drift and diffusion where the mean field game system also includes the first-order terms 
  \begin{align}\label{eq:mfg-ext}
  \left\{ 
  \begin{aligned} 
   &-\dt u - H (\nabla u) - F(\fL u)=  \cF(m)&&\text{on $\T\times\X$},\\[0.1cm]
   & \quad u(T)=\cG(m(T))&&\text{on $\X$},\\[0.2cm]
   & \dt m + \, \text{div}\, (\nabla H(\nabla u)m) - \fLs(F'(\fL u)\,m)  = 0&&\text{on $\T\times\X$},\\[0.1cm]
   &\quad m(0)=m_0&&\text{on $\X$}.
  \end{aligned}
  \right.
 \end{align}

The main challenge is to get classical solvability and strong enough a priori regularity estimates for solutions of Hamilton--Jacobi--Bellman equations. 
In the local case, the key Schauder regularity \emph{and solvability} results are proved in~\cite{MR1139064}. 
In the nonlocal case, the Schauder estimates are proved in \cite{MR3803717}, but we could find no existence result in the literature. To show existence, we adapt the continuity method described in \cite{MR3837125}, using the a priori estimates of \cite{MR3803717}, and solvability for \emph{linear} nonlocal equations of \cite{MR3201992}.
In both non-degenerate cases, uniqueness for the Fokker--Planck equation can be deduced from existing results by adapting the Holmgren method.

To summarise, the main novelties of this paper are:
\setlength{\leftmargini}{2em}
\begin{enumerate}
\smallskip
 \item  The new model and well-posedness  results in \autoref{sec:enadu}.
\item The stochastic control interpretation of the Hamilton--Jacobi--Bellman equation (in terms of time-rate change);
a heuristic derivation of \autoref{eq:mfg} in \autoref{section:stochastics}. 
 \item A theory of mean field games in $\X$ without moment assumptions (see \autoref{subsec:other}); 
the technical results in Lemmas \ref{prop:lyapunov}, \ref{lemma:V-equiv}, and \ref{lemma:approx}.
\item Existence and stability for the Fokker--Planck equation with an arbitrary L\'evy operator and a non-negative continuous coefficient in  \autoref{sec:fpe}, and their use to prove \autoref{thm:mfg-existence}, existence for \autoref{eq:mfg}.
\item Uniqueness for \autoref{eq:mfg} without strong convexity of $F$
 or strict monotonicity of $\cF$, $\cG$; 
 the 
 second half of the proof (from \eqref{eq:supp}) of \autoref{thm:mfg-uniqueness}.
 \end{enumerate}\smallskip

The paper is organized as follows:
In \autoref{sec:enadu} we introduce assumptions, solution concepts, and 
the (concrete and general) well-posedness results for the fully nonlinear mean field game system \eqref{eq:mfg} 
along an extension to system \eqref{eq:mfg-ext} with controlled drift.
The derivation of PDEs from a stochastic model is given in \autoref{section:stochastics}. 
 \autoref{sec:prelim} discusses both the background material and new results that are needed in the proofs, including tightness and approximations of L\'evy operators.
 \autoref{sec:HJB} 
 contains results on Hamilton--Jacobi--Bellman equations, 
 and in \autoref{sec:fpe_main} 
 we discuss well-posedness for Fokker--Planck equations, including existence and stability of solutions under general assumptions.
 \autoref{sec:mfg} contains the proofs of the general existence and uniqueness results for \autoref{eq:mfg}
 and can be read independently.
 Some technical proofs and auxiliary results are given in the appendices.
\end{section}

\begin{section}{Main results}\label{sec:enadu}
We present our setting and the main results regarding well-posed\-ness for \autoref{eq:mfg}, and their extension to \autoref{eq:mfg-ext}. We also discuss the lack of moment assumptions for the initial data and the L\'evy process.
\subsection{Assumptions and solution concepts}\label{subsec:ass}
 A L\'evy measure $\nu$ is defined by
 \begin{align}\label{def:levy-symmetric} \nu \ \text{is a Radon measure on $\X\setminus\{0\}$,} \quad \nu\geq0, \quad
 \int_\X \big(1\wedge |z|^2\big)\,\nu(dz)<\infty.
 \end{align}
The representation formula for the L\'evy operator $\fL$ can the be given as:
\renewcommand{\thefootnote}{\fnsymbol{footnote}}
 \begin{description*}\vspace{0.33\baselineskip}
 \item[(L)\label{L:levy}]\hspace{-0.5em}\footnotemark[2] $\fL:C^2_b(\X)\to C_b(\X)$ is a~linear operator 
with a triplet $(c,a,\nu)$, where $c\in\X$, $a\in\R^{d\times d}$,
$\nu$ is a L\'evy measure \eqref{def:levy-symmetric}, and 
\begin{multline*}
 \fL \phi(x) = c\cdot\nabla\phi(x) + \tr\big(aa^T D^2\phi(x)\big) \\
 + \int_\X \Big(\phi(x+z)-\phi(x)-\mathbbm{1}_{B_1}(z)\,z\cdot \nabla\phi(x)\Big)\,\nu(dz).
\end{multline*}
\end{description*}

 Let $\PX$ be the set of probability measures on $\X$ equipped with the topology of weak convergence of measures.
 This topology can be metrised by the Rubinstein--Kantorovich norm $\|\cdot\|_0$  defined in \autoref{def:rubinstein}.

In \autoref{eq:mfg}, we then use the following assumptions:
 \begin{description*}\vspace{0.33\baselineskip}
  \item[(A1)\label{a:F1}]\hspace{-0.5em}\footnotemark[2]
   $F\in C^1(\R)$, $F'\in\Holder{\gamma}(\R)$ for $\gamma \in (0,1]$ (see \autoref{def:holder}), and $F'\geq 0$;
  \item[(A2)\label{a:F2}] $F$ is convex;     
  \item[(A3)\label{a:m}] 
     $m_0$ is a probability measure on $\X$;
  \item[(A4)\label{a:fg1}]\hspace{-0.5em}\footnotemark[2] 
   $\cF:\CPX\to\Cb$ and $\cG:\PX\to C_b(\X)$
   are continuous, i.e.~$\lim\limits_{n\to\infty}\sup_{t\in\T}\|m_n(t) - m(t)\|_0=0$
   implies
   \begin{align*}
    \lim_{n\to\infty}\|\cF(m_n)-\cF(m)\|_\infty =0\ \ \text{and}\ \  
    \lim_{n\to\infty}\|\cG(m_n(T))-\cG(m(T))\|_\infty = 0;
   \end{align*}
  \item[(A5)\label{a:fg2}]
   $\cF$ and $\cG$ are monotone operators, namely
   \begin{align*}
    \begin{split}
     &\int_\X\big(\cG(m_1)-\cG(m_2)\big)(x)\big(m_1-m_2\big)(dx) \leq 0,\\
     &\int_0^T\int_\X\big(\cF(m_1)-\cF(m_2)\big)(t,x)(m_1-m_2)(t,dx)\,dt\leq 0,
    \end{split}
   \end{align*}
   for every pair $m_1,m_2$ in $\PX$ or $\CPX$.
  \end{description*}
    \footnotetext[2]{%
    These three conditions need to be strengthened for our results 
    to hold in the concrete cases we present,
  compare the statements of \autoref{thm:ndeg-loc}, \autoref{thm:ndeg-nloc}, and  \autoref{thm:main}.}
  \renewcommand{\thefootnote}{\arabic{footnote}}
  \begin{remark}
  \npar
  $F'\in\Holder{\gamma}(\R)$ is needed for uniqueness. For our existence results,
  $F\in C^1(\R)$  and $\gamma=0$ in \ref{a:F1} is sufficient.
  \npar $F'\geq0$ in \ref{a:F1}, but $F'\geq\kappa>0$ in 
  the concrete cases we discuss below in Sections \ref{sec:ndeg-loc} and \ref{sec:ndeg-nloc}. However, the general theory of \autoref{sec:enadu2} 
  holds under the weaker assumption $F'\geq0$. The latter results allow us to handle a class of degenerate mean field games, and are needed for the upcoming paper~\cite{IJK03b}.
  \npar By the 
  Legendre--Fenchel transform,
  for $F$ satisfying \ref{a:F1} and \ref{a:F2},
  \begin{align*}
  F(z) = \sup_{\zeta\in[0,\infty)}\big(z \zeta- F^*(\zeta) \big)
  \qquad\text{for}\qquad
  F^*(\zeta) = \sup_{z\in \R}\big(\zeta z - F(z) \big).
 \end{align*}
  Accordingly,
   \emph{every} such $F$ is the nonlinearity 
  of a  Hamilton--Jacobi--Bellman equation
  from the stochastic control theory
  \cite{MR2179357,MR2533355}. Hence for a fixed $m$, the first equation in \eqref{eq:mfg} is a  Hamilton--Jacobi--Bellman equation.\footnote{$F$ has the form of (3.2) in \cite[Chapter IV]{MR2179357} with control $v=\zeta\in U=[0,\infty)$, (drift)  $f=0$,  $a=v$, $L=F^*(v)+\cF(m)$. For $\fL=\Delta$, the first equation in \eqref{eq:mfg} is then the HJB equation (3.3) in  \cite[Chapter IV]{MR2179357}.}
  Note that $F^*(\zeta)=\infty$ for $\zeta<0$. See \autoref{appendix:legendre} and \eqref{eq:HJB}--\eqref{eq:HJB2} for more details. 
 
  \npar The operators in~\ref{a:fg1} are so-called smoothing couplings.
  Typically they are nonlocal and defined by a convolution with a fixed kernel
  (see e.g.~\cite{MR4214773}).
  \npar Assumption \ref{a:fg2} is the \emph{standard Lasry--Lions} monotonicity conditions required for unique\-ness.
  The equivalent and more familiar formulation with~$\cF$~and~$\cG$ 
  non-decreasing \cite{MR2295621,MR4214773} is obtained by
  taking $\widetilde{\cG}=-\cG$, $\widetilde{\cF}=-\cF$, and $\widetilde{u}=-u$,
  which leads to
  $-\dt \widetilde{u} = -F(-\fL \widetilde{u}) +\widetilde{\cF}(m)$
  and $\widetilde{u}(T) =\widetilde{\cG}(m)$ in 
  \autoref{eq:mfg}.
  Our choice simplifies the notation when nonlinear diffusion is involved.
  \npar We assume neither strict convexity in \ref{a:F2}
  nor strict monotonicity in \ref{a:fg2}
  and still obtain uniqueness for \autoref{eq:mfg}.
  \end{remark}
 
  With $(f,\uT) = \big(\cF(m),\cG(m(T))\big)$,
  the first pair of equations in \autoref{eq:mfg}
 form a terminal value problem for a fully nonlinear Hamilton--Jacobi--Bellman equation,
 \begin{align}\label{eq:hjb}
 \left\{
  \begin{aligned}
   -\dt u &= F(\fL u) + f\quad&\text{on $\T\times\X$},\\
   u(T)&= \uT\quad&\text{on $\X$}.
  \end{aligned}
  \right.
 \end{align}
 In this case the viscosity solution framework applies, but we consider (bounded) classical solutions, where $\dt u$ and $\fL u$ are continuous functions. Then $\fL u$  
 and the \emph{second pair} of equations in \autoref{eq:mfg} are well-defined.
  With $b=F'(\fL u)$ this pair forms
 an initial value problem for a Fokker--Planck equation,
 \begin{align}\label{eq:fp}
 \left\{
  \begin{aligned}
   &\dt m = \fLs(bm)\quad&\text{on $\T\times\X$},\\ 
   &m(0)=m_0\quad&\text{on $\X$}.
  \end{aligned}
  \right.
 \end{align}
  Since $b=F'(\fL u)$ need not be very regular and may even degenerate,\footnote{An example of a degenerate model is given in~\cite{IJK03b}.} 
  we consider very weak (measure-valued) solutions of \autoref{eq:fp}. Classical solutions $m$ would require even more regularity on $u$ and the data.
  \begin{definition}\label{def:fp-weak}
 Suppose $b\in\Cb$.
 A function $m\in\CPX$ is a \emph{very weak solution} of \autoref{eq:fp}
 if for every $\phi \in\D$ and $t\in\Tb$,
 \begin{align}\label{eq:fp-weaksolution}
 \begin{split}
 \int_\X &\phi(t,x)\,m(t,dx) - \int_\X \phi(0,x)\, m_0(dx) \\&= \int_0^t\int_\X \big( \dt \phi(\tau,x) + b(\tau,x) (\fL \phi)(\tau,x)\big)\,m(\tau,dx)\,d\tau.
 \end{split}
 \end{align}
 \end{definition}
 Now we may define the concept of solutions of \autoref{eq:mfg}.
 \begin{definition}\label{def:mfg}
  A pair $(u,m)$ is a \emph{classical--very weak solution} of \autoref{eq:mfg}
  if $u$ is a bounded classical solution of \autoref{eq:hjb} 
  with data $\big(\cF(m),\cG(m(T))\big)$,
  such that \mbox{$F'(\fL u)\in\Cb$}, and $m$ is a very weak solution of \autoref{eq:fp}
  with initial data $m_0$ and coefficient $b=F'(\fL u)$.
 \end{definition}
We now give the main results of the paper.

\subsection{Well-posedness for local second-order mean field games}\label{sec:ndeg-loc}
Here we assume $2\sigma = 2$ and:
 \begin{description*}\vspace{0.33\baselineskip}
\item[(L$^\prime$)\label{L:ndeg-loc}] 
    $\fL \phi(x) = \tr\big(aa^T D^2\phi(x)\big)\quad\text{where}\quad \det aa^T>0$.
\item[(R)\label{D:ndeg}] There are $\alpha \in (0, 1]$ and $M\in[0,\infty)$ such that the range\footnote{The space $\CaPX{\frac12}$ is defined as in \autoref{def:holder}, only with respect to the norm $\|\cdot\|_0$.}
 \begin{align*}
  \cD = \big\{\big(\cF(m),\cG(m(T))\big) :  m \in \CaPX{\frac12} \big\}
 \end{align*}
 satisfies $\cD\subset \cD_0(\alpha,M)$, where\footnote{See \autoref{def:holder_time-space}
 of spaces $\HHb{\hb}{\hd}$; \textit{BUC} = bounded uniformly continuous.}
 \begin{align*}
    \cD_0(\alpha,M)=\Big\{(f,g) : \ \ (i)\ &\  f\in \HHb{1}{\hd}, \\
 (ii)\ & \ g\in\BUC\quad\text{and}\quad\fL\uT\in L^\infty(\X),\\ (iii)\ & \ \|f\|_{1,\hd}+\|\fL \uT\|_\infty +\|\uT\|_\infty\leq M \Big\}.
\end{align*}
  \item[(A1$^{\prime}$)\label{F:ndeg}] \ref{a:F1} holds and $F'\geq\kappa$ for some $\kappa>0$ (i.e.~$F$ is strictly increasing).
 \end{description*}\smallskip
 
  Under \ref{L:ndeg-loc}, the operator $\fL$ is non-degenerate.
  For \autoref{eq:mfg} to be non-degenerate, we also need to assume \ref{F:ndeg}. Solutions $m$ always belong to $\CaPX{\frac12}$ by \autoref{lemma:fp-tightness}\autoref{item:equicontinuity}. 
    In this setting, we expect interior regularity estimates to hold.
  \begin{definition}[Interior estimates]\label{def:interior-reg}
  Assume \ref{L:levy}.
  \emph{Interior $(\hb,\hd)$-regularity estimates} hold for \autoref{eq:hjb} if for every 
   $f\in \HHb{\hb}{\hd}$, and $(t,x)\in \T\times\X$, and a viscosity solution $u$ of \autoref{eq:hjb},\footnote{
   See \autoref{def:holder_time-space}; see \autoref{def:viscosity} of viscosity solutions for $a=0$ (analogous for $a\neq0$).}
   we have
   \begin{align*}
     [\dt u]_{\HH{\hb}{\hd}{1}} + [\fL u]_{\HH{\hb}{\hd}{1}}
     \leq C(t)\big(\|f\|_{\hb,\hd} + \|u\|_\infty\big).
   \end{align*}
 \end{definition}
 
 In view of the comparison principle
 (\autoref{thm:hjb-viscosity}),
 the right-hand side 
 can be expressed in terms of $\|f\|_{\hb,\hd}$ and $\|\uT\|_\infty$.
 When $F$ is affine, interior regularity
 is given by classical Schauder theory (see e.g.~\cite{MR1406091,MR0241822,MR1465184}).
 In the fully nonlinear case, such estimates have been proved in~\cite{MR1139064}. Related results can be found in e.g.~\cite{MR661144,MR901759,MR1465184}.
  \begin{lemma}\label{lemma:hjb-local-wang}
  Assume \ref{L:ndeg-loc}, $(f,\uT)\in \cD_0(\hd,M)$ (as in \ref{D:ndeg}), \ref{F:ndeg}, \ref{a:F2}.  Then interior $(\hd/2,\hd)$-regularity estimates hold for \autoref{eq:hjb}. 
  \end{lemma}
   \begin{proof}
   The result is stated in a form which is a corollary to
    \cite[Theorem~5.2]{MR3951822}.
    As in \cite{MR3951822}, 
    it follows from the arguments in \cite{MR1139064}, in particular Theorems~1.1 and 4.13 and their proofs.
    (Our case is slightly simpler since $\fL$ is translation invariant.)
  \end{proof}
  \begin{theorem}\label{thm:ndeg-loc}
 Assume~\ref{L:ndeg-loc}, \ref{D:ndeg}, \ref{F:ndeg}, \ref{a:F2}, \ref{a:m}.
 If in addition\smallskip
  \begin{enumerate}\setlength{\itemsep}{0.33em}
      \item  \ref{a:fg1} holds,
      then there exists a classical--very weak  solution of \autoref{eq:mfg};
      \item \ref{a:fg2} holds,
      then \autoref{eq:mfg} has at most one classical--very weak solution.
\end{enumerate} 
\end{theorem}
This theorem is a corollary of the more general well-posedness result of \autoref{thm:main}.  The result and an outline of the proof is given in \autoref{sec:enadu2}.
\subsection{Well-posedness for nonlocal mean field games}\label{sec:ndeg-nloc}
 Here we assume:
 \begin{description*}\vspace{0.33\baselineskip}
  \item[(L$^{\prime\prime}$)\label{L:ndeg-nloc}] 
    Let $2\sigma\in(0,2)$ and $\fL$ be given by
    \begin{align*}
     \fL \phi(x) = \int_\X \Big(\phi(x+z)-\phi(x)
     -\mathbbm{1}_{[1,2)}(2\sigma) \mathbbm{1}_{B_1}(z)\,z\cdot\nabla\phi(x)\Big)\,\nu(dz),
   \end{align*}
   where  $\nu$ is a L\'evy measure (see \eqref{def:levy-symmetric}), $\nu|_{B_1}$ 
   is absolutely continuous with respect to the Le\-besgue measure, and
   there exists a function $k$ such that
   for $\hd$ as in \ref{D:ndeg} and $K>0$ (see \eqref{eq:holder}), 
   \begin{align*}
    \mathbbm{1}_{B_1}(z)\,\nu(dz) = \frac{k(z)}{|z|^{d+2\sigma}}\,dz,\quad
    K^{-1}\leq k(z)\leq K,\quad [k]_{\Holder{\hd}(B_1)}<\infty.
   \end{align*}   
   If $2\sigma=1$,
   then in addition $\int_{B_1\setminus B_r}\frac{zk(z)}{|z|^{d+1}}\,dz=0$ for every $r\in(0,1)$.\footnote{See \autoref{rem:triplets} for $2\sigma \in(0,1)$.}
 \item[(A1$^{\prime\prime}$)\label{a:F1'}]
   \ref{F:ndeg} holds and $F\in C^2(\R)$ and $F''\in\Holder{1}(\R)$ (see \autoref{def:holder}).
 \end{description*}\smallskip
 
 Again $\fL$ is non-degenerate and we assume \ref{F:ndeg}
 to make \autoref{eq:mfg} non-degenerate as well.
 Condition \ref{L:ndeg-nloc} defines a rich class of nonlocal operators including fractional Laplacians and the nonsymmetric operators in finance.
 There is no restriction on the tail behaviour of $\nu$ other than \eqref{def:levy-symmetric}, so underlying L\'evy processes and solutions of corresponding Fokker--Planck equations may have no moments.
 
  Despite many related results on interior regularity in the literature
 (see e.g.~\cite{MR3148110,MR3115838,MR3803717,MR3667677,MR1432798,MR3951822}),
 we could not find a statement we could cite.
In \autoref{appendix:comparison} we therefore prove the following.
 \begin{lemma}\label{lemma:regularity}
  Assume~\ref{L:ndeg-nloc}, $(f,\uT)\in \cD_0(\hd,M)$ (as in \ref{D:ndeg}),~\ref{a:F1'},~\ref{a:F2}.
   Then interior $(\frac{\hd}{2\sigma},\hd)$-regularity estimates hold for \autoref{eq:hjb}.
 \end{lemma} 
We expect the result to be true under weaker regularity assumptions on $F$.

 \begin{theorem}\label{thm:ndeg-nloc}
 Assume~\ref{L:ndeg-nloc},~\ref{D:ndeg}, \ref{a:F1'}, \ref{a:F2}, \ref{a:m}.
 If in addition \smallskip
  \begin{enumerate}
      \item  \ref{a:fg1} holds,
      then there exists a classical--very weak  solution of \autoref{eq:mfg};
      \item \ref{a:fg2} holds, then \autoref{eq:mfg} has at most one classical--very weak solution.
  \end{enumerate} \end{theorem}

  This theorem is a corollary of the more general well-posedness result of \autoref{thm:main}.  The result and an outline of the proof is given in \autoref{sec:enadu2}.
 
  \subsection{General well-posedness theory}\label{sec:enadu2}
  We describe the properties of solutions to \autoref{eq:hjb} and \autoref{eq:fp}
 that lead to well-posedness of \autoref{eq:mfg}.
  Let 
  \begin{align*}
      \HJcD &= 
      \big\{\text{$u\in\Cbb$ is a bounded classical solution of \autoref{eq:hjb}}\\ &\qquad\text{with data $(f,\uT)=\big(\cF(m),\cG(m(T))\big)$}:  m\in\CaPX{\frac12}\big\},
      \\
  \cB &= \big\{F'\big(\fL u\big): u\in \HJcD\big\}.    
  \end{align*}
  
  \begin{description*}\vspace{0.33\baselineskip}
   \item[(S1)\label{R1}] For every $m\in\CaPX{\frac12}$ there exists a
   bounded classical solution $u$ of \autoref{eq:hjb} with data $(f,\uT)=\big(\cF(m),\cG(m(T))\big)$.\smallskip
   \item[(S2)\label{R2}] If $\{u_n,u\}_{n\in\N}\subset \HJcD$
   are such that $\lim\limits_{n\to\infty}\|u_n- u\|_\infty = 0 $,
   then $\fL u_n(t) \to \fL u(t)$ uniformly on compact sets in $\X$ for every $t\in\T$.\smallskip
   \item[(S3)\label{R3}] There exists $\KHJ\geq0$ such that 
   $\|F'(\fL u)\|_{\infty}\leq \KHJ$ for every $u\in \HJcD$.\smallskip
  \item[(S4)\label{U1}]
    It holds
   $\{\dt u,\,\fL u : u\in \HJcD\}\subset\Cb$.\smallskip
  \item[(S5)\label{U2}] For each $b\in\cB\cap\Cb$  
   and initial data $m_0\in\PX$ there exists at most one very weak solution of \autoref{eq:fp}.
 \end{description*}\smallskip
 
  Condition~\ref{R1} describes existence of solutions of the Hamilton--Jacobi--Bellman equation \eqref{eq:hjb}, 
 which are unique by 
 \autoref{thm:hjb-viscosity},
 and~\ref{U2} describes uniqueness of solutions of the Fokker--Planck equation~\eqref{eq:fp},
 which exist by
 \autoref{thm:fp-existence}.
 Conditions~\ref{R2}, \ref{R3}, \ref{U1} describe various (related) properties of solutions of \autoref{eq:hjb}.
  Under \ref{a:F1}, both \ref{R3}  and  \ref{U1} imply $b=F'(\fL u)\in\Cb$ for $u\in \HJcD$. 
  
 \begin{theorem}\label{thm:main}
  Assume~\ref{L:levy},~\ref{a:F1},~\ref{a:m}.
  If in addition\smallskip
  \begin{enumerate}\setlength{\itemsep}{0.33em}
      \item \ref{a:fg1}, \ref{R1}, \ref{R2}, \ref{R3} hold,
      then there exists a classical--very weak  solution of \autoref{eq:mfg};
      \item \ref{a:F2},~\ref{a:fg2}, \ref{U1}, \ref{U2} hold,
      then \autoref{eq:mfg} has at most one classical--very weak solution.
  \end{enumerate}
 \end{theorem}
\subsubsection*{Proof of \autoref{thm:main}}
The results are proved in \autoref{sec:mfg}.
Existence is addressed in~\autoref{thm:mfg-existence}
by an application of the Kakutani--Glicksberg--Fan fixed point theorem,
which requires a detailed analysis
of \autoref{eq:hjb} and \autoref{eq:fp}.
Of particular interest are the compactness and stability
results \autoref{lemma:fp-tightness}, \autoref{cor:m-compact},
and \autoref{lemma:fp-continuity} for the Fokker--Planck equation.

 Uniqueness follows by \autoref{thm:mfg-uniqueness}.
 Note that in contrast to previous work  (cf.~e.g.~\cite[(1.24), (1.25)]{MR4214773}) we only need (non-strict) convexity of $F$ in~\ref{a:F2} and (non-strict) monotonicity of $\cF$ and $\cG$ in~\ref{a:fg2}, without further restrictions.
\subsubsection*{Proofs of Theorems \ref{thm:ndeg-loc}, \ref{thm:ndeg-nloc}} The well-posedness results for non-degenerate cases (Sections \ref{sec:ndeg-loc}--\ref{sec:ndeg-nloc}) follow by verifying the general conditions \ref{R1}--\ref{U2} and  then applying \autoref{thm:main}.
 We obtain \ref{R1}--\ref{U2}
 from \autoref{thm:hjb-interior}
 and \autoref{thm:fp-uniqueness} ---
 see \autoref{cor:hjb-l2-s} and \autoref{cor:b-uniq}.
 \subsection{Extensions to include controlled drift}\label{sec:extension}
 When the diffusion operator $\fL$ is non-degenerate and of order $2\sigma>1$ (local or nonlocal), the well-posedness results above can easily be extended to include controlled drift. To illustrate this, we consider \autoref{eq:mfg-ext} which comes from a model where the drift and the time-rate changes of the driving L\'evy process are controlled separately (with separate controls). 
  \begin{theorem}\label{thm:ndeg-ext}
 Assume~\ref{L:ndeg-loc} or \ref{L:ndeg-nloc} with $\sigma>\frac12$, and \ref{D:ndeg}, \ref{a:F1'},  \ref{a:F2}, \ref{a:m}. Let $H \in C^{2}(\X)$ be strictly convex and $D^2 H \in \Holder{1}(\X,\X\times\X)$.
 If in addition\smallskip
\begin{enumerate}\setlength{\itemsep}{0.33em}
      \item  \ref{a:fg1} holds,
      then there exists a classical--very weak  solution of \autoref{eq:mfg-ext};
      \item \ref{a:fg2} holds,
      then \autoref{eq:mfg-ext} has at most one classical--very weak solution.
\end{enumerate} 
\end{theorem}

 We omit the details of the proof.
By the assumed regularity and convexity of $H$, it can be adapted from the
arguments we use to prove \autoref{thm:ndeg-loc} or \autoref{thm:ndeg-nloc}. With little additional difficulty, most of the effort would involve tedious rewriting of the results of \autoref{sec:fpe}.
Importantly, the results we used from \cite{MR1139064,MR3803717,MR3201992} to establish the interior regularity estimates for the Hamilton--Jacobi--Bellman equation, as well as uniqueness for the Fokker--Planck equation, still hold in the setting of \autoref{thm:ndeg-ext}.  
\subsection{Mean field games in \texorpdfstring{$\X$}{Rd} without moment assumptions}\label{subsec:other}
 In the mean field game literature (see e.g.~\cite{MR4214773}),
 it is common to use the Wasserstein-1 space $(\mathcal{P}_1,d_1)$  (or Wasserstein-$p$ for $p>1$)
 in the analysis of the Fokker--Planck equations. Here $\mathcal{P}_1$ is the space of probability measures with finite first moments.
 For compactness, finite $1+\epsilon$ moments are typically assumed.

 Moments of solutions of the Fokker--Planck equation depend on both the driving L\'evy process
 and the initial distribution.
 L\'evy processes have the same moments
 as the tails of their L\'evy measures \cite[Theorem 25.3]{MR3185174},
 e.g.~the Brownian motion has moments of any order,
 while a $2\sigma$-stable process with $\nu(dz)\approx\frac{dz}{|z|^{d+2\sigma}}$
 only has moments of order less than $2\sigma\in(0,2)$.
 Condition \ref{L:ndeg-nloc} puts no restriction on $\nu|_{B_1^c}$.
 The mean field games we consider may thus be driven by processes with unbounded first moments,
 like the $2\sigma$-stable processes for $2\sigma\leq1$.
 This means that we cannot work in $(\mathcal{P}_1,d_1)$,
 even when the initial distribution $m_0$ has moments of all orders.

 We work in the space $(\mathcal{P},d_0)$ of probability measures under weak convergence,
 metrised by $d_0$, defined by the Rubinstein--Kantorovich norm $\|\cdot\|_0$
 (see \autoref{subsec:meas}).
 The $d_0$-topology is strictly weaker than the $d_1$-topology,
 as it does not require convergence of first moments.
 The tools we develop can be useful for other problems 
 and have already been used  \cite{MR4309434,chowdhury2021numerical,jakobsen2023master}. In the local case they yield results for a larger class of initial distributions $m_0$ than usually considered.
 Crucial ingredients are more refined tightness arguments
 and their interplay with L\'evy processes. In particular,
 the sequence of Lemmas \ref{prop:lyapunov}, \ref{lemma:V-equiv} and \ref{lemma:approx}, regarding compact sets in $(\mathcal{P},d_0)$, and \emph{a priori} estimates for approximations of L\'evy operators, leading to \autoref{lemma:fp-tightness}.
 
\end{section}
\begin{section}{Derivation of the model}\label{section:stochastics}
 In this section we show heuristically that \autoref{eq:mfg} is related to a mean field game
 where players control the time change rate of a L\'evy process.
 Random time change of SDEs is a well-established
 technique~\cite{MR3363697,MR1011252,MR455113,MR1046331} with applications
 e.g.~in modelling markets or turbulence~\cite{MR2484103,CARR2004113}.
For SDEs driven by self-similar processes, like the Brownian motion or an $\alpha$-stable process, this type of control coincides with the classical (continuous) control \cite{MR2179357,MR2322248}.\footnote{By self-similarity (e.g.~for the Brownian motion $B_{ct} = \sqrt{c} B_t$) controlled time change is equivalent to control of the strength of the diffusion (controlled diffusion).}
 However, for other L\'evy processes, including compound Poisson
 and most jump processes used in finance and insurance, this is not the case.

 This type of a control problem seems to be new and we plan to analyse it in full detail in a future paper.
 
\subsection{Time changed L\'evy process}
 We start by fixing a L\'evy process $X_t$ and the filtration $\{\mathcal{F}_t\}$ it generates.
 The infinitesimal generator $\fL$ of $X$ is given by~\ref{L:levy}.
  \begin{definition}[{\cite[Definition~1.1]{MR3363697}}]
 A \emph{random time change} $\rtc_s$ is an almost surely
 non-negative, non-decreasing stochastic process which is a finite stopping
 time for each fixed $s$.\footnote{$\rtc_s$ is a stopping time if
 $\{\rtc_s\leq \tau\}\subset \mathcal{F}_\tau$ for $\tau\geq 0$.}
 It is \emph{absolutely continuous} if there exists a non-negative
 $\mathcal{F}_s$-adapted process $\rtc'$ such that $\rtc(s) -\rtc(0)= \int_0^s\rtc'(\tau)\,d\tau$.
 \end{definition}
 For $(t,x)\in\T\times\X$ and $s\geq t$, we define an $\mathcal{F}_s$-adapted
 L\'evy process $X^{t,x}_s$ starting from $X^{t,x}_t=x$  by
  $X_s^{t,x}= x+ X_s-X_t$.
 Then, for an absolutely continuous random time change $\rtc_s$ such that $\rtc_t=t$,
 $\rtc_{s}'$~is deterministic at $s=t$,
 and $\rtc_{s+h}-\rtc_{s}$ is independent of $\mathcal{F}_{\rtc_s}$
 for all $s,h\geq0$, we define a time-changed process 
     $Y^{t,x,\rtc}_s = X_{\rtc_s}^{t,x}$.
 It is an inhomogeneous Markov process associated with the families of operators~$P^\rtc$
 and transition probabilities~$p^\rtc$ (see~\cite[\S1.1, \S1.2~(10)]{MR2058260}) given by
 \begin{align}\label{eq:P}
  P_{t,s}^{\rtc}\phi(x) = \int_\X \phi(y)\,p^{\rtc}(t,x,s,dy) = E\phi\big(Y^{t,x,\rtc}_{s}\big)
 \end{align}
 for $\phi\in C_b(\X)$.
 To compute the ``generator'' $\fL_\rtc$ of $Y^{t,x,\rtc}$,
 note that by the Dynkin formula~\cite[(1.55)]{MR3156646}, if $\phi \in \Dom(\fL)$
 \begin{align*}
  E\phi\big(Y^{t,x,\rtc}_{s}\big) -\phi(x)=
  E\bigg(\int_t^{\rtc_s}\fL \phi(X^{t,x}_\tau)\,d\tau\bigg),
  \end{align*}
  and by a change of variables,
  \begin{align*}
   \frac{P_{t+h,t}^{\rtc}\phi(x)-\phi(x)}{h}=\frac{E\phi\big(Y^{t,x,\rtc}_{t+h}\big) - \phi(x)}{h}
   =E\bigg(\frac1h\int_t^{t+h}\fL\phi\big(X^{t,x}_{\rtc_\tau}\big)\rtc'_\tau\,d\tau\bigg).
  \end{align*}
 Under some natural assumptions, we can show that
$X^{t,x}_{\rtc_\tau}\to x$ as $\tau\to t$ and use the dominated convergence theorem etc.~to get that
  \begin{align}\label{eq:afl}
  \fL_\rtc\phi(x)=\lim_{h\to0^+}\frac{P_{t+h,t}^{\rtc}\phi-\phi}{h}(x)
  = \rtc'_t\,\fL\phi(x).
 \end{align}
 A proof of a more general result can be found
 in e.g.~\cite[Theorem~8.4]{MR3363697}.

\subsection{Control problem and Bellman equation}\label{sec:bellman} To control the
process $Y_s^{t,x,\rtc}$, we introduce a running gain (profit, utility) $\ell$, a
terminal gain $\uT$, and an expected total gain functional 
 \begin{align*}
  J(t,x,\rtc) = E\bigg(\int_t^T
  \ell\big(s,Y_s^{t,x,\rtc},\rtc_s'\big)\,ds+\uT\big(Y_T^{t,x,\rtc}\big)\bigg).
 \end{align*}
The goal is to find an
admissible control $\rtc^*$ that maximizes $J$.
If such a control exists, the optimally controlled process is given by $Y^{t,x,\rtc^*}_s\!$.

Under a suitable definition of the set of admissible controls
$\mathcal{A}$ and standard
assumptions on $\ell$ and $\uT$, $J$ is well-defined.
The corresponding value function $u$ (the optimal value of $J$) is given by
 \begin{align}\label{eq:JJ}
  u(t,x) = \sup_{\rtc\in\mathcal{A}} J(t,x,\rtc).
 \end{align}

Let $h>0$ and $t+h<T$.
By the dynamic programming principle,
\begin{align*}u(t,x)  = \sup_{\rtc} E\bigg(\int_t^{t+h}
  \ell\big(s,Y^{t,x,\rtc}_s,\rtc'_s\big)\,ds +
  u\big(t+h,Y^{t,x,\rtc}_{t+h}\big)\bigg),
  \end{align*}
  and hence
\begin{align*}
  &-\frac{u(t+h,x)-u(t,x)}{h}\\
  &\quad=\sup_{\rtc} E\bigg(\frac{u\big(t+h,Y^{t,x,\rtc}_{t+h}\big)-u(t+h,x)}h
  +\frac1h\int_t^{t+h}\ell\big(s,Y^{t,x,\rtc}_s,\rtc'_s\big)\,ds \bigg).
  \end{align*}
Recalling the definition of $\fL_\rtc$ in \eqref{eq:afl}, we can
(heuristically at least) pass to the limit as $h\to 0$ and find the
following dynamic programming --- or Bellman --- equation
\begin{align}\label{eq:HJB}
  -\dt u = \sup_{\zeta\geq0}\Big(\zeta \fL u + \ell(t,x,\zeta) \Big),
  \end{align}
satisfied e.g.~in the viscosity sense (see \autoref{sec:HJB}),
where $\zeta$  denotes  the (deterministic) value of $\rtc'_t$ to simplify the notation.
 We now assume that
  $\ell(t,x,\zeta) = -L(\zeta)  + f(t,x)$,
  where $L:[0,\infty)\to\R\cup\{\infty\}$ is a convex,
  lower-semicontinuous function.
  Then the Bellman equation can be expressed in terms
  of the Legendre--Fenchel transform~$F$ of~$L$,
  i.e.~$F(z)=\sup_{\zeta\geq0}\big(\zeta z-L(\zeta)\big)$, as
  \begin{align}\label{eq:HJB1}
  -\dt u = F\big(\fL u\big) + f(t,x).
  \end{align}
  By the definitions of $u$ and $X_T^{T,x}$ it also follows that
  \begin{align}\label{eq:HJB2}
    u(T,x) = E\uT\big(X_T^{T,x}\big)=\uT(x).
  \end{align}
  
\subsection{Optimal control and Fokker--Planck equation}\label{sec:optctrl}
 By the properties of the Legendre--Fenchel transform,
 when $\lim\limits_{\zeta\to\infty}L(\zeta)/\zeta = \infty$ and $L$ is strictly convex on $\{L\neq\infty\}$,
 the optimal value $\zeta$ in \eqref{eq:HJB} satisfies
 $\zeta=F'(\fL u)$ for every $(t,x)\in\T\times\X$ (see \autoref{prop:G-F}).
 We therefore obtain a function
\begin{align}\label{eq:fb}
  b(t,x) =\zeta = (\rtc^*)'_t = F'\big(\fL u(t,x)\big).
   \end{align}
This is the optimal time change rate in the feedback form.
The optimally controlled process and the optimal control in~\eqref{eq:JJ} are then implicitly given by
\begin{align*}
Y^*_s=X_{\rtc^*_s}^{t,x} \qquad \text{and}\qquad \rtc_s^*=t+\int_t^s
b(\tau,Y^*_\tau)\,d\tau.
\end{align*}
They are well-defined if $b$ is e.g.~bounded and
continuous.

  By  defining $p^{\rtc^*}(t,x,s,A)=\mathsf{P}(Y_s^*\in A)$, 
  if solutions of equations \eqref{eq:HJB1}--\eqref{eq:HJB2} are  unique,
  we obtain a unique family of transition probabilities $p^{\rtc^*}$ (cf.~\eqref{eq:P}),
   satisfying the Chapman--Kolmogorov relations.
  This family, in turn, defines a wide-sense Markov process
  (see~\cite[\S1.1 Definition~1]{MR2058260}).
  Given an initial condition $m(0) = m_0\in\PX$,
  the (input) distribution $m$ of this Markov process
  (see~\cite[\S1.1 Definition~3]{MR2058260})\footnote{Alternatively,
  $m(t)$ is the distribution of the solution $Z(t)$ of 
  SDE $dZ(t) = b(t, Z(t))\, dX(t)$, $Z(0) \sim m_0$.
  Moreover, $Y_s^* = E\big[Z(s)| Z(t) = x\big]$,  see \cite[\S1.2\,(9),\,(10)]{MR2058260}.}
  satisfies
 \begin{align*}
  \int_\X \varphi(x)\,m(t+h,dx) = \int_\X\int_\X \varphi(y)\,p^{\rtc^*}(t,x,t+h,dy)\,m(t,dx),
 \end{align*}
 for every $\varphi\in C_c^\infty(\X)$ and $t,h\geq 0$.
 Then,
 \begin{multline*}
  \int_\X \big(\varphi(t,x)\,m(t,dx)-\varphi(t+h,x)\,m(t+h,dx)\big) \\
  = \int_\X\int_\X \big(\varphi(t,y)\,p^{\rtc^*}(t,x,t,dy)-\varphi(t+h,y)p^{\rtc^*}(t,x,t+h,dy)\big)\,m(t,dx),
 \end{multline*}
 and because of~\eqref{eq:afl},~\eqref{eq:fb} and the fact that $p^{\rtc^*}(t,x,t,dy)=\delta_x(dy)$,
 this leads to 
  \begin{align*}
  \dt \int_\X \varphi(t,x)\,m(t,dx) = \int_\X \Big(b(t,x)\fL\varphi+\dt\varphi(t,x)\Big)\,m(t,dx).
 \end{align*}
 Since $b=F'(\fL u)$, by  duality (see \autoref{def:fp-weak}) $m$ is a very weak solution of
 \begin{align}\label{eq:FPE}
  \dt m = \fLs \big(F'(\fL u)\, m\big),\quad m(0) = m_0,
 \end{align}
 where $\fLs$ is the formal adjoint of $\fL$.

\subsection{Heuristic derivation of the mean field game} A mean field
 game is a limit of games between identical players as the
 number of players tends to infinity.
 In our case, each player controls the time change rate of her own independent copy
 of the L\'evy process $X$, with running and terminal gains depending on
 the anticipated distribution $\widehat m$ of the processes controlled (optimally) by the other
 players (see~\ref{a:fg1})
 \begin{align*}
   f=\cF(\widehat m)\qquad\text{and}\qquad
   \uT=\cG\big(\widehat m(T)\big).
   \end{align*}
 By the results of \autoref{sec:bellman} the corresponding 
Bellman equation for
each player is
\begin{align*}
\left\{\begin{aligned}
  -\dt u &= F(\fL u) + \cF(\widehat m)\quad &\text{on $\T\times\X$},\\
  u(T)&= \cG\big(\widehat m(T)\big)\quad &\text{on $\X$}.
\end{aligned}\right.
\end{align*}
 Note that the solution $u$ depends on $\widehat m$,
 and then so does the optimal feedback control~\eqref{eq:fb}.
 Suppose that the players’ processes start from some known initial distribution $m_0\in\PX$.
 Then, the actual distribution $m$ of their optimally controlled processes
 is given by the solution of the Fokker--Planck equation \eqref{eq:FPE},
 described in \autoref{sec:optctrl}.

At a Nash equilibrium we expect $\widehat m=m$, i.e.~the anticipations of the players to be correct.
The result is a closed model of coupled equations as in \autoref{eq:mfg}.
\end{section}

\begin{section}{Preliminaries}\label{sec:prelim}
  By $K_d= 2\pi^{d/2}\Gamma(d/2)^{-1}$ we denote the surface measure of the $(d-1)$-dimen\-sional unit sphere.
  By $B_r$ and $B_r^c$ we denote the ball of radius $r$ centred at $0$
  and its complement in $\X$.
  Similarly, $B_r(x)$ denotes a ball centred at $x$.

 \begin{definition}\label{def:holder}
  A function $\phi$ is H\"older-continuous at $x\in \X$ with parameter $\hd\in(0,1]$ if
  for some $r>0$
  \begin{align}\label{eq:holder}
   [\phi]_{\Holder{\hd}(B_r(x))} = \sup_{y\in B_r(x)\setminus\{x\}} \frac{|\phi(x)-\phi(y)|}{|x-y|^\hd}<\infty.
  \end{align}
  The space $\Holder{\hd}(\X)$ consists of functions which are H\"older-continuous
  at every point in $\X$ with parameter $\hd$.
  Further, define $\Hb{\hd} = \{\phi:\|\phi\|_{\hd}<\infty\}$, where
  \begin{align*}
      [\phi]_{\hd} = \sup_{x\in \X}\,[\phi]_{\Holder{\hd}(B_1(x))} 
      \quad \text{and}\quad \|\phi\|_{\hd} = \|\phi\|_{L^\infty(\X)} + [\phi]_{\hd}.
  \end{align*}
  \end{definition}
  Note that the definition of $\Hb{\hd}$ is equivalent 
 to the more standard notation, where the supremum in \eqref{eq:holder} is taken over $|x-y|\in\X\setminus\{0\}$.
 The space $\Hb{1}$ consists of bounded, Lipschitz-continuous functions.
 By $C^1(\X)$, $C^2(\X)$ we denote spaces of once or twice continuously differentiable functions.
 \begin{definition}\label{def:holder_time-space}
  For $(t,x)\in\T\times\X$ and $\hd,\hb\in(0,1]$, define 
 \begin{align*}
    [\phi]_{\HH{\hb}{\hd}{r}} =
    \sup_{y\in B_r(x)}\, [\phi(y)]_{\Holder{\hb}([0,t])} +
    \sup_{s\in [0,t]}\, [\phi(s)]_{\Holder{\hd}(B_r(x))}.
 \end{align*}
 \end{definition}
 We also denote
 $\mathcal{C}_b^{\hb,\hd}([0,t]\times\X) = \{\phi:\|\phi\|_{\mathcal{C}^{\hb,\hd}([0,t]\times\X)}<\infty\}$,
 where
 \begin{align*}
     \|\phi\|_{\mathcal{C}^{\hb,\hd}([0,t]\times\X)}
     = \|\phi\|_{L^\infty([0,t]\times\X)}+\sup_{x\in\X}[\phi]_{\HH{\hb}{\hd}{1}}.
 \end{align*}
 
  \begin{definition}\label{def:bounded}
   When $X$ is a normed space, $B(\T,X)$ denotes the space of bound\-ed functions from $\T$ to $X$, i.e.~$B(\T,X) = \big\{u:\T\to X\ :\ \textstyle\sup_{t\in\T}\|u(t)\|_X<\infty\big\}$.
  \end{definition} 
  Note the subtle difference between $B(\T,X)$ and the usual space $L^\infty(\T,X)$.
  
\subsection{Spaces of measures}\label{subsec:meas}

 Let $\PX$ consist of probability measures on $\X$, a~subspace of
 the space of bounded Radon measures $\Mb=C_0(\X)^*$.
 Denote
 \begin{align*}
  m[\phi] = \int_\X \phi(x)\,m(dx)\quad\text{for every $m\in\PX$ and $\phi\in C_b(\X)$.}
 \end{align*}
 The space $\PX$ is equipped with the topology of weak convergence of measures,\!\!~
 \footnote{It is also called \emph{narrow}, \emph{vague} or \emph{weak-$*$} convergence.}
 \begin{align*}
  \lim_{n\to\infty}m_n = m \quad\text{if and only if}\quad \lim_{n\to\infty}m_n[\phi] 
  = m[\phi]\text{ for every $\phi\in C_b(\X)$}.
 \end{align*}
 This topology can be metrised by an embedding into a normed space (see~\cite[\S8.3]{MR2267655}).
 \begin{definition}\label{def:rubinstein}
  The Rubinstein--Kantorovich norm $\|\cdot\|_0$ on $\Mb$ is given by 
  \begin{align*}
   \|m\|_0 = \sup \big\{m[\psi]: \psi \in\Hb{1},\ \|\psi\|_\infty\leq1,\ [\psi]_1\leq 1\big\}.
  \end{align*}
 \end{definition}
 While the space $\big(\Mb,\|\cdot\|_0\big)$ is not completely metrisable,
 thanks to~\cite[Theorems~4.19 and~17.23]{MR1321597}, both $\PX$ and $\CPX$ are complete spaces. Let
  \begin{align*}
  \mathcal{P}_{ac}(\X) = \big\{u\in L^1(\X):\|u\|_{L^1(\X)} = 1,\ u\geq0\ \big\}= L^1(\X)\cap\PX.
 \end{align*}
 We endow $\mathcal{P}_{ac}(\X)$ with the topology inherited from $\PX$.
 \begin{definition}
  A set of measures $\Pi\subset\PX$ is tight
  if for every $\epsilon>0$ there exists a compact set $K_\epsilon\subset\X$
  such that for every $m\in\Pi$ we have $m(K_\epsilon) \geq 1-\epsilon$.
 \end{definition}
 This concept is important because of the Prokhorov theorem, which states that
  a set $\Pi\subset\PX$ is pre-compact if and only if it is tight.
 \begin{definition}\label{def:lyapunov}
  A real function $V\in C^2(\X)$ is a Lyapunov function if $V(x) = V_0\big(\sqrt{1+|x|^2}\big)$
  for some  subadditive, non-decreasing function $V_0:[0,\infty)\to[0,\infty)$
  such that $\|V_0'\|_\infty,\|V_0''\|_\infty\leq 1$, and $\lim\limits_{x\to\infty}V_0(x)= \infty$.
 \end{definition}
 \begin{remark}\label{rem:lyapunov} 
 \npar\label{rem:bd-der} Because $\|V_0'\|_\infty,\|V_0''\|_\infty\leq 1$,
  we have $\|\nabla V\|_\infty,\| D^2V\|_\infty\leq 1$.
  Note that the choice of the constant $1$ in this condition is arbitrary.
 \npar $\big(1+|x|^2\big)^{a/2}$ for $a\in(0,1]$ and $\log\big(\sqrt{1+|x|^2}+1\big)$
  are Lyapunov functions.
 \npar If $m_0\in\PX$ has a finite first moment and~$V$ is any Lyapunov function,
  then $m_0[V]<\infty$.
  Indeed, since $0\leq V_0'\leq 1$, we have $V(x)\leq V(0)+|x|$, thus $m_0[V]\leq V(0)+\int_\X |x|\,dm_0$.
 \end{remark}
  
 \begin{proposition}\label{prop:pre-compactness}
 If $V$ is a Lyapunov function, then for every $r>0$ the set
  \begin{align*}
   \mathcal{P}_{V,r} = \big\{m\in\PX : m[V] \leq r\big\}
  \end{align*}
  is tight and then compact by the Prokhorov theorem.
 \end{proposition}
 \begin{proof}
  Notice that the set $\mathcal{P}_{V,r}$ is closed.
  Let $\epsilon>0$.
  Since $\lim\limits_{|x|\to\infty}V(x) = \infty$, the set $K_\epsilon = \{x:V(x)\leq \frac{r}{\epsilon}\}$ is compact.
  Then it follows from the Chebyshev inequality that for every $m\in\mathcal{P}_{V,r}$,
  \begin{align*}
   m\big(K_\epsilon^c\big)\ \leq\ \frac{\epsilon}{r}\int_{\{V > \frac{r}{\epsilon}\}}V\,dm
   \ \leq\ \frac{\epsilon}{r}m[V]\ \leq\ \epsilon.
  \end{align*}
  Hence the set $\mathcal{P}_{V,r}$ is tight and thus compact by the Prokhorov theorem.
 \end{proof}
 The reverse statement is also true.
 \begin{lemma}\label{prop:lyapunov}
  If the set $\Pi\subset\PX$ is tight,
  then there exists a Lyapunov function~$V$ such that $m[V]\leq 1$ for every $m\in\Pi$.
 \end{lemma}
 This result is crucial for our paper and is the reason why our findings 
 hold without moment assumptions.
 The proof is given in \autoref{appendix:comparison}.
 
\subsection{L\'evy operators} In this section we collect some basic observations on L\'evy operators. Recall the representation formula given in \ref{L:levy} in \autoref{subsec:ass}.
\begin{remark}\label{rem:triplets}
If $\int_{B_1} |z|\,\nu(dz)<\infty$, then we may equivalently write 
\begin{align*}
 \fL \phi = \bigg(c-\int_{B_1}z\,\nu(dz)\bigg)\cdot\nabla\phi + \tr\big(aa^T D^2\phi\big) 
 + \int_\X \Big(\phi(x+z)-\phi(x)\Big)\,\nu(dz).
\end{align*}
In particular, we may have $\big(\int_{B_1}z\,\nu(dz),0,\nu\big)$ as a triplet in \ref{L:levy}. 
\end{remark}
 \begin{lemma}\label{lemma:V-equiv}
  Assume~\ref{L:levy} and $V$ is a Lyapunov function.
  The following are equivalent\smallskip
  \begin{enumerate}
 \begin{minipage}{0.36\linewidth}
    \item\label{item:V} $\int_{B_1^c}V(z)\,\nu(dz) <\infty$;
    \item\label{item:fL} $\|\fL V\|_\infty <\infty$;
   \end{minipage}
   \begin{minipage}{0.63\linewidth}
   \vspace{4pt}
    \item\label{item:th1} $\vartheta_1(x) = \int_{B_1^c}\big(V(x+z)-V(x)\big)\,\nu(dz)\in L^\infty(\X)$;
    \item\label{item:th2} $\vartheta_2(x) = \int_{B_1^c}\big|V(x+z)-V(x)\big|\,\nu(dz)\in L^\infty(\X)$.
   \end{minipage}
  \end{enumerate}
 \end{lemma}
 \begin{proof}
 Let 
  $$\vartheta_0(x) = c\cdot \nabla V(x) + \tr\big(aa^T D^2 V(x)\big) 
  + \int_{B_1}\!\! \Big(V(x+z)-V(x)-z\cdot \nabla V(z)\Big)\nu(dz).$$
 Because $V$ is a Lyapunov function (see \autoref{rem:lyapunov}\autoref{rem:bd-der}), we have 
 \begin{align*}
  \|\vartheta_0\|_\infty\leq |c| + |a|^2+ \int_{B_1}|z|^2\,\nu(dz).
 \end{align*}
 Observe that 
 $\|\fL V\|_\infty -\|\vartheta_0\|_\infty \leq \|\vartheta_1\|_\infty \leq \|\vartheta_2\|_\infty$,
 hence \!\autoref{item:th2}$\,\Rightarrow$\autoref{item:th1}$\,\Rightarrow$\autoref{item:fL}.
 We also notice $\|\fL V\|_\infty \geq \|\vartheta_1\|_\infty-\|\vartheta_0\|_\infty$
 and $\int_{B_1^c} V(z)\,\nu(dz) = \vartheta_1(0)+\nu(B_1^c)$,
 thus \autoref{item:fL}$\,\Rightarrow$\autoref{item:th1}$\,\Rightarrow$\autoref{item:V}.

 It remains to prove \!\autoref{item:V}$\,\Rightarrow$\autoref{item:th2}.
 Let $V_0\big(\sqrt{1+|x|^2}\big) = V(x)$ as in \autoref{def:lyapunov} and notice that,
 because $V_0$ is subadditive and non-decreasing, we have
 \begin{align*}
  \big|V(y)-V(x)\big| \leq V_0\Big(\big|\sqrt{1+|y|^2}-\sqrt{1+|x|^2}\big|\Big) 
  \leq V_0\big(\sqrt{1+|y-x|^2}\big).
 \end{align*}
 Now we may estimate
 \begin{align*}
  \int_{B_1^c}\big|V(x+z)-V(x)\big|\,\nu(dz)\leq \int_{B_1^c} V(z)\,\nu(dz).
 \end{align*}
 \end{proof}
 \begin{corollary}\label{rem:levy-lyapunov}
  Assume~\ref{L:levy},~\ref{a:m}.
  There exists a Lyapunov function $V$ such that $m_0[V],\|\fL V\|_\infty <\infty$.
 \end{corollary}
 \begin{proof}
  Since $\nu|_{B_1^c}$ is a bounded measure,
  the set $\{\nu|_{B_1^c},m_0\}$ is tight.
  Hence, by \autoref{prop:lyapunov} we may find a Lyapunov function
  such that $\int_{B_1^c}V(z)\,\nu(dz) <\infty$ and $m_0[V]<\infty$.
  Thanks to \autoref{lemma:V-equiv}\autoref{item:fL} we also have $\|\fL V\|_\infty <\infty$.
 \end{proof}
 
 Let $\fL$ be a L\'evy operator with triplet $(c,a,\nu)$.
 Denote
 \begin{align}\label{eq:lk}
    \|\fL\|_{\text{\textit{LK}}} = |c| + |a|^2 +  \frac12 \int_{B_1} |z|^2\,\nu(dz) + 2 \nu(B_1^c).
 \end{align}\vspace{-\baselineskip}
\begin{proposition}\label{prop:lx}
Assume~\ref{L:levy}, $\phi \in C^2_b(\X)$. Then
    $\|\fL \phi\|_{\infty} \leq   \|\fL\|_{\text{\textit{LK}}} \|\phi\|_{C^2_b(\X)}$.
\end{proposition}
\begin{proof}
Using the Taylor expansion, we calculate
\begin{align*}
&\|\fL \phi\|_{\infty} \leq  |c|\|\nabla\phi\|_{\infty} + |a|^2 \|D^2 \phi\|_{\infty} \\
&\quad+ \Big|\int_\X \big(\phi(x+z)-\phi(x) - \mathbbm{1}_{B_1}(z)\,z\cdot \nabla\phi(x)\big)\,\nu(dz) \Big| \\
&\leq  |c|\|\nabla\phi\|_{\infty} + |a|^2 \|D^2 \phi\|_{\infty} 
+ \frac{\|D^2\phi\|_{\infty}}2 \int_{B_1} |z|^2\,\nu(dz) + 2\|\phi\|_{\infty} \nu(B_1^c).
\end{align*}
\end{proof}
\begin{remark}
The mapping $\fL\mapsto \|\fL\|_{\text{\textit{LK}}}$ is a norm on the space (convex cone) of L\'evy operators.
It dominates the operator norm $C^2_b(\X)\to C_b(\X)$, but they are not equivalent.
\end{remark}
 \begin{lemma}\label{lemma:approx}
  Assume~\ref{L:levy}.
  For $\epsilon\in(0,1)$ there exist $\fL^\epsilon$, $\nu^\epsilon$ such that
  \begin{align}\label{eq:fl-1}
   \fL^\epsilon\mu(x) = \int_\X\big(\mu(x+z) -\mu(x)\big) \, \nu^{\epsilon}(dz),
  \end{align}
  where \mbox{$\fL^\epsilon: L^1(\X)\to L^1(\X)$}, $\nu^\epsilon(\X)<\infty$ and $\supp\nu^\epsilon \subset \X\setminus B_\epsilon$.
  Moreover,\smallskip
  \begin{enumerate}
      \item\label{item:eps-cube}
       $\|\fL^\epsilon \mu\|_{L^1(\X)}\leq \big({c_\fL}/{\epsilon^3}\big) \|\mu\|_{L^1(\X)}$
       for a~constant $c_\fL>0$;
      \item\label{item:fl-convergence}
       $\lim\limits_{\epsilon\to0}\|\fL^\epsilon\varphi -\fL\varphi\|_\infty = 0$
       for every $\varphi\in C_c^\infty(\X)$;
      \item\label{item:fl-bound}
       $\sup_{\epsilon\in(0,1)}\big(\|\fL^\epsilon V\|_\infty+\|\fL^\epsilon\|_{\text{\textit{LK}}}\big)<\infty$
       for every Lyapunov function $V$ such that \mbox{$\|\fL V\|_\infty<\infty$}.
  \end{enumerate}
 \end{lemma}
 \begin{proof*}
 \begin{part*}
 Let $(c,a,\nu)$ be the L\'evy triplet of $\fL$ and $a= (a_1, \ldots, a_d) \in\R^{d\times d}$
 with $a_i \in\X$.
 Consider $\nu^{\epsilon} = \nu_c^\epsilon+\nu_a^\epsilon+\nu_1^\epsilon+\nu_2^\epsilon$, where
 \begin{align*}
 \nu_c^\epsilon &= \frac{|c|}{\epsilon}\delta_{\epsilon\frac{c_{}}{|c|}}, 
 &\nu^\epsilon_1(E) &= \nu(E\setminus B_{\epsilon}),\\
 \nu_a^\epsilon &= \sum_{i=1}^d\frac{|a_i|^2}{\epsilon^2}(\delta_{\epsilon \frac{a_i}{|a_i|}}+\delta_{-\epsilon \frac{a_i}{|a_i|}}),\;\;
 &\nu^\epsilon_2(E) &= \frac{1}{\epsilon}\nu\Big(\big(B_{1}\setminus B_{\epsilon}\big) \cap (-E/\epsilon)\Big),
 \end{align*}
 Notice that $\nu^{\epsilon}$ is a bounded, non-negative measure
 with $\supp\nu^\epsilon\subset \X\setminus B_\epsilon$
 (hence a L\'evy measure).
 Let $\fL^\epsilon = \fL^\epsilon_{\loc}+\fL^\epsilon_{\nloc}$, where, for $\mu\in L^1(\X)$,
 \begin{multline*}
  \fL^\epsilon_{\loc}\mu(x) = \int_\X\big(\mu(x+z)-\mu(x)\big)\,(\nu_c^\epsilon+\nu_a^\epsilon)(dz) \\
  = \frac{|c|}{\epsilon}\Big({\mu\big(x+\epsilon \tfrac{c}{|c|}\big)-\mu(x)}\Big)
   +\sum_{i=1}^d\frac{|a_i|^2}{\epsilon^2}\Big(\mu\big(x+\epsilon \tfrac{a_i}{|a_i|}\big)+\mu\big(x-\epsilon \tfrac{a_i}{|a_i|}\big)-2\mu(x)\Big).
 \end{multline*}
 and
 \begin{align*}
  \fL^\epsilon_{\nloc}\mu &=  \int_\X \big(\mu(x+z)-\mu(x)\big)\,(\nu_1^\epsilon+\nu_2^\epsilon)(dz)\\  
  &=\int_{B_\epsilon^c}\bigg(\mu(x+z)-\mu(x)
   +\mathbbm{1}_{B_1}(z)\frac{\mu(x-\epsilon z)-\mu(x)}{\epsilon}  \bigg)\,\nu(dz).
 \end{align*}
 Note that
 \begin{align*}
  \nu_1^\epsilon(B_1\setminus B_\epsilon)+\nu_2^\epsilon(\X)=(1+\epsilon^{-1})\nu(B_1\setminus B_{\epsilon})
  \leq (\epsilon^{-2}+\epsilon^{-3})\int_{B_1}|z|^2\,\nu(dz),
 \end{align*}
 and hence
 \begin{align*}
     \|\fL^\epsilon\mu\|_{L^1(\X)}&\leq \bigg(\frac{2|c|}{\epsilon}+\frac{4|a|^2}{\epsilon^2}+2\nu(B_1^c)
       + \frac{2+2\epsilon}{\epsilon^3} \int_{B_1}|z|^2\,\nu(dz)\bigg)\|\mu\|_{L^1(\X)}\\
     &\leq \frac{4}{\epsilon^3}\bigg(|c|+|a|^2+\int_\X\big(1\wedge|z|^2\big)\,\nu(dz)\bigg)\|\mu\|_{L^1(\X)}.
 \end{align*}
 This shows that $\fL^\epsilon: L^1(\X)\to L^1(\X)$ 
 and $\|\fL^\epsilon \mu\|_{L^1(\X)}\leq \big({c_\fL}/{\epsilon^3}\big) \|\mu\|_{L^1(\X)}$.
 \end{part*}
 \begin{part*}
 For every $\varphi\in C_c^\infty(\X)$, by using the Taylor expansion and the Cauchy--Schwarz inequality (for the third-order remainder), we get
 \begin{align}\label{eq:estimate-local-approx}
    \Big|\Big(\fL^\epsilon_{\loc}-c\cdot\nabla-\tr\big(aa^T D^2(\,\cdot\,)\big)\Big)\varphi(x)\Big|
    \leq\epsilon\bigg(\frac{|c|}{2}\|D^2 \varphi\|_{\infty}
    +|a|^2\|D^3 \varphi\|_{\infty}  \bigg).
 \end{align}
 Let 
  $\fL_\nu \varphi (x) = \int_\X \big( \varphi(x+z) -\varphi(x) 
  - \mathbbm{1}_{B_1}(z)\,z\cdot \nabla\varphi(x) \big)\,\nu(dz)$.
 Then
 \begin{align}\label{eq:estimate-nonlocal-approx}
  \begin{split}
    \Big|\big(\fL^\epsilon_{\nloc} - \fL_\nu\big)\varphi(x) \Big| 
   & = \bigg| \int_{B_1\setminus B_\epsilon} \bigg(\frac{\varphi(x-\epsilon z) -\varphi(x)}{\epsilon} 
   + z\cdot \nabla\varphi(x) \bigg)\,\nu(dz) \\
   &\qquad\qquad-\int_{B_\epsilon}\Big(\varphi(x+z)-\varphi(x)-z\cdot\nabla\varphi(x)\Big)\,\nu(dz)\bigg| \\
   & \leq \frac{\epsilon}{2} \|D^2\varphi\|_{\infty} \int_{B_1} |z|^2\,\nu(dz) 
   + \frac{1}{2} \|D^2 \varphi\|_{\infty} \int_{B_\epsilon} |z|^2\,\nu(dz).
  \end{split}
 \end{align}
 Since $\lim\limits_{\epsilon\to0}\int_{B_\epsilon} |z|^2\,\nu(dz)=0$
 by the Lebesgue dominated convergence theorem,
 it follows from~\eqref{eq:estimate-local-approx} and~\eqref{eq:estimate-nonlocal-approx} that 
  $\lim_{\epsilon \to 0}\|(\fL^\epsilon -\fL)\varphi\|_{\infty} =0$.
\end{part*}
\begin{part*}
 Let $V$ be a Lyapunov function such that $\|\fL V\|_\infty <\infty$.
 Then also $\|\fL_\nu V\|_\infty <\infty$.
 By the definition of $\fL^\epsilon=\fL^\epsilon_{\loc}+\fL^\epsilon_{\nloc}$,
 in a way similar to \eqref{eq:estimate-local-approx}, \eqref{eq:estimate-nonlocal-approx},
 \begin{align*}
  \|\fL^\epsilon V\|_\infty 
  \leq |c|\|\nabla V\|_\infty  +  |a|^2 \|D^2 V\|_\infty + \|D^2 V\|_\infty \int_{B_1} |z|^2\,\nu(dz)
  + \|\fL_\nu V\|_\infty.
 \end{align*}
 Thus $\sup_{\epsilon\in(0,1)}\|\fL^\epsilon V\|_\infty <\infty$.
 Notice that
 \begin{align*}
  \int_{B_1} z \,\nu_c^\epsilon(dz)=c,
  \quad\int_{B_1} z \,\nu_a^\epsilon(dz)=0,
  \quad\text{and}
  \quad\int_{B_1} z \,(\nu_1^\epsilon+\nu_2^\epsilon)(dz) = 0,
 \end{align*}
 thus the L\'evy triplet of the operator $\fL^\epsilon$ is $(c,0,\nu_\epsilon)$
 (see \autoref{rem:triplets}).
 Hence
 \begin{align*}
     \|\fL^\epsilon\|_{\text{\textit{LK}}} & = |c| + \frac{\epsilon|c|}{2} + |a|^2
     + \frac12\int_{B_1\setminus  B_\epsilon} (1+\epsilon)|z|^2\,\nu(dz) +2\nu(B_1^c)\\
     &\leq (1+\epsilon)\|\fL\|_{\text{\textit{LK}}}.
 \end{align*}
 \end{part*}
 \vspace{-2\baselineskip}
 \end{proof*}
 \end{section}
\begin{section}{Hamilton--Jacobi--Bellman equations}\label{sec:HJB}
 In this section we define viscosity solutions and give results for \autoref{eq:hjb}.
 Let $(t,x,\ell)\mapsto \mathcal{F}\big(t,x,\ell)$ and $w_0$ be continuous functions,
 and $\mathcal{F}$ be non-decreasing in $\ell$.
 For $\fL$ satisfying~\ref{L:levy} with $a=0$,\footnote{We take $a=0$
 for simplicity and to use the results of \cite{MR3592657}.}
 consider the following problem
 \begin{align}\label{eq:viscosity}
  \left\{\begin{aligned}
   \dt w&= \mathcal{F}\big(t,x,(\fL w)(t,x)\big),\quad&\text{on $\T\times\X$},\\
   w(0) &= w_0,\quad&\text{on $\X$}.
  \end{aligned}\right.
 \end{align}
 For $0\leq r<\infty$ and $p\in\X$ we introduce linear operators
 \begin{align*}
 \begin{split}
  \fLhigh (\phi,p)(x) 
  &= \int _{B_r^c}\Big(\phi(x+z)-\phi(x)-\mathbbm{1}_{B_1}(z)\,z\cdot p\Big)\,\nu(dz),\\
  \fLlow \phi(x)  &= \int _{B_r} \Big(\phi(x+z)-\phi(x)-
  \mathbbm{1}_{B_1}(z)\,z\cdot\nabla\phi(x)\Big)\,\nu(dz),
  \end{split}
 \end{align*}
 defined for
 bounded semicontinuous and $C^2$ functions respectively.
 
 \begin{definition}\label{def:viscosity}
  A bounded upper-semicontinuous function $u^-:\Tb\times\X\to\R$ is a \textit{viscosity subsolution}
  of \autoref{eq:viscosity} if 
   $u^-(0,x)\leq w_0(x)$ for every $x\in\X$
   and for every $r\in(0,1)$, test function $\phi\in C^2\big(\T\times\X\big)$,
   and a maximum point $(t,x)$ of $u^--\phi$,
   \begin{align*}
    \dt \phi(t,x) - \mathcal{F}\Big(t,x,\big(c\cdot \nabla\phi 
    + \fLhigh \big(u^-,\nabla\phi(t,x)\big)+\fLlow\phi\big)(t,x)\Big) \leq 0.
   \end{align*}
 \end{definition}
  A supersolution is defined similarly, replacing max, upper-semicontinuous, and ``$\leq$'' by min, lower-semicontinuous, and ``$\geq$''. A viscosity solution is a sub- and supersolution at the same time. Note that bounded classical solutions are also bounded viscosity solutions.
\begin{definition}\label{def:comparison}
  The comparison principle holds for \autoref{eq:viscosity}
  if any subsolution $u^-$ and supersolution $u^+$
  satisfy $u^-(t,x)\leq u^+(t,x)$ for every  $(t,x)\in\Tb\times\X$.
 \end{definition}
We have the following uniqueness, stability, and existence result for viscosity solutions of \autoref{eq:hjb}.
 \begin{theorem}\label{thm:hjb-viscosity}
  Assume~\ref{L:levy},~\ref{a:F1}, and $(f,\uT)$
  are bounded and continuous.\smallskip
   \begin{enumerate}
\item The comparison principle (see \autoref{def:comparison}) holds for \autoref{eq:hjb}.
\item\label{item:cp} Let $u_1,u_2$ be viscosity solutions of \autoref{eq:hjb}
  with bounded uniformly continuous data $(f_1,\unT{1})$, $(f_2,\unT{2})$, respectively.
  Then for every $t\in\Tb$,
    \begin{align*}\hspace{-.5\mathindent}
     \|u_1(t)-u_2(t)\|_{\infty} \leq (T-t) \|f_1-f_2\|_{\infty} + \|\unT{1} - \unT{2}\|_{\infty}.
    \end{align*} 
\item\label{item:vs-ex} There exists a unique viscosity solution of \autoref{eq:hjb}.   
\end{enumerate}
  \end{theorem}
 \begin{proof}
 \begin{part*}\label{part:cp}
 In the nonlocal case ($a=0$) with uniformly continuous $f,u_0$, this is \cite[Theorem~6.1]{MR3592657}. In the general case the result follows from a standard but long and tedious combination of the arguments of \cite{MR3592657} and \cite{MR2129093}. We omit this proof.
 \end{part*}
 \begin{part*}
  Note that for $\{i,j\}=\{1,2\}$,
    \begin{align*}v_i(t,x) = u_j(t,x) - (T-t)\|f_1-f_2\|_{\infty} - \|\unT{1}-\unT{2}\|_{\infty}\end{align*}
  is a viscosity subsolution of \autoref{eq:hjb} with data $(f_i,\unT{i})$. The result then follows from the comparison principle in \autoref{part:cp}. 
\end{part*}
\begin{part*}
 \autoref{part:cp} entails uniqueness of viscosity solutions.
 It also implies existence of solutions through the Perron method (cf.~\cite[Section 4]{MR1118699}). See also \cite[Theorems~6.2]{MR3592657} for the result when $a=0$.
 \vspace{-14pt}\[\ \]\vspace{-\belowdisplayskip}
 \end{part*}
 \end{proof}\smallskip

 Now we give results that are specific for non-degenerate cases of
 \autoref{eq:hjb}, which correspond to the setting of Sections \ref{sec:ndeg-loc} and \ref{sec:ndeg-nloc}.
 \begin{proposition}\label{lemma:hjb-dt-bound}
   Assume \ref{L:levy}, 
  \ref{a:F1}, and $u$ is a viscosity solution of \autoref{eq:hjb}
  with bounded uniformly continuous data $(f,\uT)$ such that
  $\dt f \in L^\infty\big(\T\times\X\big)$ and $\fL \uT\in L^\infty(\X)$.
  Then $\dt u \in L^\infty\big(\T\times\X\big)$ and 
  \begin{align*}
     \|\dt u(t)\|_\infty \leq (T-t)\|\dt f\|_\infty + \|F(\fL \uT)\|_\infty +\|f\|_\infty.
  \end{align*}
 \end{proposition}
 \begin{proof}
   Take $h>0$ and $\uT_\epsilon = \uT*\rho_\epsilon$, where $\rho_\epsilon$ is the standard mollifier.
   Note that $v_\epsilon(t,x)=\uT_\epsilon(x)$ is a viscosity (classical) solution of \autoref{eq:hjb}
   with data $(-F(\fL\uT_\epsilon),\uT_\epsilon)$,
   hence by \autoref{thm:hjb-viscosity}\autoref{item:cp},
   \begin{align*}
       \|u(T-h)-\uT\|_\infty \leq h\|F(\fL\uT_\epsilon)+f\|_\infty + 2\|\uT_\epsilon-\uT\|_\infty.
   \end{align*}
   By \ref{a:F1}, $\|F(\fL\uT_\epsilon)\|_\infty\leq \|F(\fL\uT)\|_\infty$, and because $\uT\in\BUC$,
   $\|\uT_\epsilon-\uT\|_\infty$ can be arbitrarily small.
   Thus,
   \begin{align*}
       \|u(T-h)-u(T)\|_\infty \leq h\big(\|F(\fL\uT)\|_\infty +\|f\|_\infty\big).
   \end{align*}
   
   Similarly, $v_h(t,x) = u(t-h,x)$ is a viscosity solution of \autoref{eq:hjb} with
   data $(f(\,\cdot\,-h),u(T-h))$, thus for every $t\in\T$,
   \begin{align*}
       \|u(t)-v_h(t)\|_\infty &\leq (T-t)\|f(\,\cdot\,)-f(\,\cdot\,-h)\|_\infty + \|u(T-h)-u(T)\|_\infty\\
       &\leq (T-t)\|\dt f\|_\infty h + \|F(\fL\uT)+f\|_\infty h.
   \end{align*}
  Hence $u$ is Lipschitz in time.
  \end{proof}
 \begin{theorem}\label{thm:hjb-interior}
 Assume \ref{L:ndeg-loc} or \ref{L:ndeg-nloc}, $(f,\uT)\in \cD_0(\hd,M)$ (as in \ref{D:ndeg}), \ref{F:ndeg}, \ref{a:F2}, and
 interior $(\frac{\hd}{2\sigma},\hd)$-regularity estimates (\autoref{def:interior-reg}) hold for \autoref{eq:hjb}.\smallskip 
  \begin{enumerate}
      \item\label{item:unique-local}
      There exists a bounded classical solution $u$ of \autoref{eq:hjb}.
      \item\label{item:convergence-local}
      If $u_n$ are bounded classical solutions of \autoref{eq:hjb} with data 
      $(f_n, \unT{n})\in \cD_0(\hd,M)$
      and  \mbox{$\lim\limits_{n\to\infty}\|u_n-u\|_{\infty}=0$},
      then $\fL u_n(t) \to\fL u(t)$ uniformly on compact sets in $\X$ for every $t\in\T$.
      \item\label{item:uniform-cont-local}
     $\dt u,\,\fL u \in\Cb$ 
     and for every $t\in\T$ there is a constant $C(t,f,\uT)$ such that 
     $\|\fL u\|_{\mathcal{C}^{\hd/2\sigma,\hd}([0,t]\times\X)}\leq C(t,f,\uT)$.
  \end{enumerate}
 \end{theorem}
 \begin{proof}
\begin{part*}\label{part:unique-local}
  There exists a bounded viscosity solution by \autoref{thm:hjb-viscosity} \autoref{item:vs-ex}. 
  Because of the interior regularity estimates,
  we have $\dt u,\,\fL u\in C(\T\times \X)$,
  hence $u$ is a bounded classical solution of \autoref{eq:hjb}.
  \end{part*} 
  \begin{part*}
  By \autoref{part:unique-local} and interior regularity estimates, for every $t\in \T$ and $r>0$,
  there exists a constant $C(t,r)>0$ such that 
  \begin{align*}
   \sup_n \Big(\|\fL u_n(t)\|_{L^\infty(B_r)} + [\fL u_n (t) ]_{\Holder{\hd}(B_r)}\Big)  \leq C(t,r).
  \end{align*}
  By the Arzel\`a--Ascoli theorem,
  for every $t\in \T$ there exist a subsequence $\{u_{n_k}\}$ and a function $v\in C_b(\X)$ such that
  $\fL u_{n_k}(t) \to v$ uniformly on compact sets in $\X$.
  For $\varphi\in C_c^\infty(\X)$, we note that
  \begin{align*}
    \lim_{k\to\infty}\int_\X \fL u_{n_k}(t,x)\varphi(x)\,dx 
    = \int_\X v(x)\varphi(x)\,dx,
  \end{align*}
   and since 
  $\lim_{n\to\infty}\|u_{n}-u\|_\infty = 0$ and $\fLs\varphi\in L^1(\X)$,
  \begin{align*}
  \lim_{k\to\infty}\int_\X \fL u_{n_k}(t,x)\varphi(x)\,dx 
  &=\lim_{k\to\infty}\int_\X u_{n_k}(t,x)\fLs\varphi(x)\,dx\\
  &= \int_\X u(t,x)\fLs\varphi(x)\,dx
  = \int_\X \fL u(t,x) \varphi(x)\,dx.
  \end{align*}
  Hence $v(x)=\fL u(t,x)$, and $\fL u_{n_k}(t) \to \fL u(t)$ uniformly on compact sets in $\X$
  for every $t\in\T$.
  \end{part*}
  \begin{part*}
  By \autoref{part:unique-local} and \autoref{lemma:hjb-dt-bound},
  $\dt u \in C_b (\T\times \X)$.
  Since $u$ is a bounded classical solution and $F'\geq\kappa$,
  we also have $\fL u = F^{-1}(-\dt u-f)\in C_b(\T\times \X)$.
  Moreover,
   $\|\fL u\|_\infty \leq F^{-1}\big(T\|\dt f\|_\infty + \|F(\fL \uT)\|_\infty + 2\|f\|_\infty\big)$.
  
  By \autoref{thm:hjb-viscosity}\autoref{item:cp}, we have $\|u\|_\infty \leq T\|f\|_\infty + \|\uT\|_\infty$.
  Thus, by interior regularity estimates
  (which are uniform in $x$, see \autoref{def:interior-reg}),
  for every $t\in\T$,
  \begin{align*}
    \|\fL u\|_{\mathcal{C}^{\hd/2\sigma,\hd}([0,t]\times\X)} 
    &\leq \|\fL u\|_\infty + \sup_{x\in\X}\Big([\fL u]_{\HH{\hd/2\sigma}{\hd}{1}}\Big)\\
    &\leq \widetilde{C}(t) \Big(\|f\|_{\hd/2\sigma,\hd}+\|\dt f\|_\infty
     +\|\fL \uT\|_\infty +\|\uT\|_\infty\Big).
  \end{align*}
  \vspace{-14pt}
\end{part*}
\end{proof}
\begin{corollary}\label{cor:hjb-l2-s}
   Assume \ref{L:ndeg-loc} or \ref{L:ndeg-nloc}, and \ref{D:ndeg}, \ref{F:ndeg}, \ref{a:F2}.
   If interior $(\frac{\hd}{2\sigma},\hd)$-regularity estimates hold for \autoref{eq:hjb},
  then \ref{R1}, \ref{R2}, \ref{R3}, \ref{U1} are satisfied.
 \end{corollary}
 \begin{proof}
 Condition \ref{R1} follows from \autoref{thm:hjb-interior}\autoref{item:unique-local},
 while \ref{R2} follows from  \autoref{thm:hjb-interior}\autoref{item:convergence-local}, and
 \ref{R3}, \ref{U1} hold by  \autoref{thm:hjb-interior}\autoref{item:uniform-cont-local}.
 \end{proof}
 \begin{remark}
 If instead of \ref{D:ndeg} we only assume $\cD\subset \HHb{\hd/2\sigma}{\hd}\times\BUC$ (uniformly bounded in an appropriate way) in \autoref{cor:hjb-l2-s}, then we still obtain \ref{R1}~and \ref{R2}.
 We may get \ref{R3} by assuming $F'\leq K$
 (i.e.~$F$~is globally Lip\-schitz).
 This is enough for existence in \autoref{thm:mfg-existence},
 but not for uniqueness in \autoref{thm:mfg-uniqueness}.
 \end{remark}
\end{section}
\begin{section}{Fokker--Planck equations}\label{sec:fpe_main}
\subsection{Existence}\label{sec:fpe}
 In this section we prove existence for \autoref{eq:fp}.
 We assume:
 \begin{description*}\vspace{0.33\baselineskip}
  \item[(B)\label{a1':b}]$b\in\Cu$ and $b(t,x)\in[0,B]$
  for fixed $B\in[0,\infty)$ and every $(t,x)\in\T\times\X$.
 \end{description*}\smallskip
 
 For $b=F'(\fL u)$ this is a consequence of~\ref{a:F1} and either~\ref{R3} or \ref{U1}
 when $u$ is a bounded classical solution of \autoref{eq:hjb}.
   \begin{lemma}\label{lemma:defn-fp-sol}
   Let $m\in\CPX$ and $m(0)=m_0$.
   The following are equivalent\smallskip
 \begin{enumerate}
     \item\label{item:D} $m$ is a very weak solution of \autoref{eq:fp} (cf.~\autoref{def:fp-weak});
     \item\label{item:D-large} $m$ satisfies \eqref{eq:fp-weaksolution} for every 
     \begin{align*}\hspace{-.5\mathindent}
      \phi\in\mathcal{U} = \big\{\phi\in\Cbb:\ \dt \phi+ b\fL \phi\in\Cb\big\};
     \end{align*}
     \item\label{item:D-small} $m$ satisfies \eqref{eq:fp-weaksolution} for every\,\footnote{In this set functions are constant in time.} 
     \begin{align*}\hspace{-.5\mathindent}
      \phi\in\big\{\phi\in\D : \phi(t) = \psi\in C_c^\infty(\X) \text{ for every }{t\in\Tb}\big\}.
     \end{align*}
 \end{enumerate}
 \end{lemma}
 \begin{proof}
 Implications 
 \!\autoref{item:D-large}$\,\Rightarrow$\autoref{item:D}$\,\Rightarrow$\autoref{item:D-small}
 are trivial.
 By a density argument we get \!\autoref{item:D}$\,\Rightarrow$\autoref{item:D-large}.
 To prove \!\autoref{item:D-small}$\,\Rightarrow$\autoref{item:D},
 fix $\varphi\in\D$, $t\in\Tb$, and consider a sequence of simple functions
 $\varphi^k 
 = \sum_{{n}=1}^{N_k}\mathbbm{1}_{[t_{n}^k,t_{n+1}^k)}\varphi(t_{n}^k)\stackrel{k}{\to}\varphi$  pointwise, 
 where $\bigcup_{n}[t_{n}^k,t_{n+1}^k)=[0,t)$ for each $k\in\N$ and $t_{n}^k <t_{n+1}^k$.
 Then by \!\autoref{item:D-small} we have
 \begin{align*}
     \sum_{{n}=1}^{N_k} \big(m(t_{n+1}^k) - m(t_{n}^k)\big)[\varphi(t_{n}^k)] 
     = \sum_{{n}=1}^{N_k}\int_{t_n^k}^{t_{n+1}^k}m(\tau)\big[b(\tau)\fL\varphi(t_{n}^k)\big]\,d\tau.
 \end{align*}
 Notice that by the Lebesgue dominated convergence theorem we get
 \begin{multline*}
     \lim_{k\to\infty}
     \sum_{{n}=1}^{N_k}\int_{t_n^k}^{t_{n+1}^k}m(\tau)\big[b(\tau)\fL\varphi(t_{n}^k)\big]\,d\tau\\
     = \lim_{k\to\infty}\int_0^t m(\tau)\big[b(\tau)\fL\varphi^k(\tau)\big]\,d\tau 
     = \int_0^t m(\tau)\big[b(\tau)\fL\varphi(\tau)\big]\,d\tau.
 \end{multline*}
 We also observe that
 \begin{multline*}
     \sum_{n=1}^{N_k} \big(m(t_{n+1}^k) - m(t_{n}^k)\big)[\varphi(t_{n}^k)] \\
     = m(t)[\varphi(t)] - m_0[\varphi(0)] 
     - \sum_{{n}=1}^{N_k} \Big(m(t_{n+1}^k)[\varphi(t_{n+1}^k) - \varphi(t_{n}^k)]\Big).
 \end{multline*}
 By the Taylor expansion, for some $\xi_{n}^k\in[t_{n}^k,t_{n+1}^k]$ we have
 \begin{align*}
    \varphi(t_{n+1}^k) - \varphi(t_{n}^k) 
    = \dt \varphi(t_{n+1}^k)(t_{n+1}^k-t_{n}^k) - \dt^2\varphi(\xi_{n}^k)\frac{(t_{n+1}^k-t_{n}^k)^2}{2}.
 \end{align*}
 Since $m\in\CPX$, by considering the relevant Riemann integral on $[0,t]$, we get
 \begin{align*}
     \lim_{k\to\infty}\sum_{{n}=1}^{N_k} \Big(m(t_{n+1}^k)[\varphi(t_{n+1}^k) - \varphi(t_{n}^k)]\Big) 
     = \int_0^t m(\tau)[\dt \varphi(\tau)]\,d\tau.
 \end{align*}
 By combining these arguments we obtain
  \begin{align*}
   m(t)[\varphi(t)] 
   = m_0[\varphi(0)] + \int_0^t m(\tau)\big[\dt\varphi(\tau) 
   + b(\tau) \big(\fL \varphi(\tau)\big)\big] d\tau.
  \end{align*}
 \end{proof}
 \begin{lemma}\label{lemma:fp-tightness}
  Assume triplets $(\fL_\lambda,b_{\lambda},m_{0,\lambda})_\lambda$
  satisfy~\ref{L:levy}, \ref{a1':b},~\ref{a:m} for each~$\lambda$,
  and let $\FPm_\lambda$ be the sets of very weak solutions of problems
  \begin{align*}
  \left\{
  \begin{aligned}
   &\dt m_\lambda = \fLs_\lambda(b_\lambda m_\lambda)\quad&&\text{on \ $\T\times\X$},\\
   &m_\lambda(0)=m_{0,\lambda}\quad&&\text{on \   $\X$}.
  \end{aligned}
  \right.
  \end{align*}
  If $\bigcup_\lambda\big\{m_{0,\lambda},(\nu_\lambda)|_{B_1^c}\big\}$ is tight and
  $\sup_\lambda\big(\|b_\lambda\|_\infty  +  \|\fL_\lambda\|_{\text{\textit{LK}}}\big)
  <\infty$,\footnote{See~\eqref{eq:lk} for the definition of $\|\cdot\|_{\text{\textit{LK}}}$.}
  then\smallskip
  \begin{enumerate}
   \item\label{item:tightness} for every $\epsilon>0$
   there exists a compact set $K_\epsilon\subset\X$ such that
  \begin{align*}\hspace{-.5\mathindent}
   \sup \Big\{\sup_{t\in\Tb}m(t)(K_\epsilon^c):
   m\in\textstyle\bigcup_{\lambda}\FPm_\lambda\Big\}\leq\epsilon;
  \end{align*}
   \item\label{item:equicontinuity} for every $m\in\bigcup_{\lambda}\FPm_\lambda$ we have 
  \begin{align*}\hspace{-.5\mathindent}
   \|m(t)-m(s)\|_0 \leq \sup_\lambda 
   \Big(2 + \big(2\sqrt{T}+K_d\big)\|b_\lambda\|_\infty\|\fL_\lambda\|_{\text{\textit{LK}}}\Big) \sqrt{|t-s|};
   \end{align*}
   \item\label{item:compactness} the set $\bigcup_{\lambda}\FPm_\lambda\subset\CPX$ is pre-compact.
  \end{enumerate}
 \end{lemma}
 \begin{proof}
 \begin{part*}\label{part:compactness}
  Let $V(x) = V_0\big(\sqrt{1+|x|^2}\big)$ be a Lyapunov function for which we have
  \mbox{$\sup_\lambda\big(m_{0,\lambda}[V]+\|\fL_\lambda V\|_\infty\big)<\infty$}
  (see \autoref{prop:lyapunov}, \autoref{lemma:V-equiv}, \autoref{rem:levy-lyapunov}).
  For $n\in\N$, let $V_{n,0}\in C^2_b\big([0,\infty)\big)$ be such that
  \begin{align*}
      V_{n,0}(t)=\left\{\begin{aligned} &V_0(t) &&\text{for $t\leq n$},\\
       &V_0\big(\sqrt{1+(n+1)^2}\big)\quad &&\text{for $t\geq n+2$},\end{aligned}\right.
  \end{align*}
  and additionally 
  \begin{align}\label{eq:Vn-est}
  0\leq V_{n,0}'\leq V_0'\qquad\text{and}\qquad |V_{n,0}''|\leq |V_0''|.
  \end{align}
  Take $V_n(x) = V_{n,0}\big(\sqrt{1+|x|^2}\big)$.
  Thanks to \autoref{lemma:defn-fp-sol}, for every $m\in\FPm_\lambda$,
  \begin{align}\label{eq:very_weak_soln-Vn}
   m(t)[V_n] = m_{0,\lambda}[V_n]+\int_0^tm(\tau)[b_\lambda(\tau)\fL_\lambda V_n]\,d\tau.
  \end{align}
  Notice that $|V_n(x)-V_n(y)|\leq |V(x)-V(y)|$ and
  \begin{align}\label{eq:Vn-limit}
      \lim_{n\to\infty}\big(V_n,\nabla V_n, D^2 V_n\big)(x)
      = \big(V,\nabla V, D^2 V\big)(x)\quad\text{for every $x\in\X$.}
  \end{align} 
  We now use the formula in~\ref{L:levy} with $\phi=V_n$
  and separate the integral part on domains $B_1$ and $B_1^c$.
  Because of \eqref{eq:Vn-limit}, by the Lebesgue dominated convergence theorem --- 
  we use \autoref{lemma:V-equiv}\autoref{item:th2} for the integral on $B_1^c$
  and \eqref{eq:Vn-est} otherwise ---
  we may pass to the limit in \eqref{eq:very_weak_soln-Vn}.
  For every $t\in\Tb$, $\lambda$, and $m\in\FPm_\lambda$ we obtain
  \begin{align}\label{eq:m-bound}
   m(t)[V] = m_{0,\lambda}[V]+\int_0^tm(\tau)[b_\lambda\fL_\lambda V]\,d\tau
   \leq m_{0,\lambda}[V] + \|b_\lambda\|_\infty \|\fL_\lambda V\|_\infty T.
  \end{align}
  Thus, by \autoref{prop:pre-compactness},
  for every $\epsilon>0$ there exists a~compact set $K_\epsilon$ such that
  \begin{align*}
   \sup \Big\{m(t)(K_\epsilon^c):
   t\in\Tb,\ m\in\textstyle\bigcup_{\lambda}\FPm_\lambda\Big\}\leq\epsilon.
  \end{align*}
\vspace{-\baselineskip}
\end{part*}
\begin{part*}\label{part:equicontinuity}
  Consider $\phi_{\epsilon} = \phi * \rho_{\epsilon}$,
  where $\phi\in\Hb{1}$ is such that $\|\phi\|_\infty\leq 1$ and $[\phi]_1\leq1$,
  and $\rho_{\epsilon}$ is a standard mollifier.
  Then $\|\phi - \phi_{\epsilon}\|_{\infty} \leq \epsilon$ and, by \autoref{prop:lx},
  $\|\fL \phi_{\epsilon}\|_\infty\leq \|\fL\|_{\text{\textit{LK}}}\|\phi_\epsilon\|_{C^2_b(\X)}$.
  By \autoref{def:fp-weak}, for every $\lambda$ and $m\in\FPm_\lambda$,
  \begin{align*}
   \big|\big(m(t)-m(s)\big)[\phi] \big| 
   &=  \big|\big(m(t)-m(s)\big)[\phi-\phi_{\epsilon}]
   + \big(m(t)-m(s)\big)[\phi_{\epsilon}] \big|  \\
   &\leq 2\epsilon + \bigg|\int_s^t\int_\X (\fL_\lambda \phi_\epsilon)(x) 
   b_\lambda(\tau,x)\,m(\tau,dx)\,d\tau\bigg|\\
   &\leq 2\epsilon + \|b_\lambda\|_\infty\|\fL_\lambda\|_{\text{\textit{LK}}} \|\phi_\epsilon\|_{C^2_b(\X)} |t-s|.
  \end{align*}
  We also have
  \begin{align*}
  \|\phi_\epsilon\|_{C^2_b(\X)} \leq \bigg(\|\phi\|_\infty+\|\nabla\phi\|_\infty
  +\frac{K_d\|\nabla\phi\|_{\infty}}{\epsilon}\bigg)\leq  \frac{2\epsilon+K_d}{\epsilon}.
  \end{align*}
  By taking $\epsilon = \sqrt{|t-s|}$, we thus obtain
  \begin{align*}
   \|m(t)-m(s)\|_0 
   \leq \sup_\lambda\Big(2 + \big(2\sqrt{T}+K_d\big)\|b_\lambda\|_\infty\|\fL_\lambda\|_{\text{\textit{LK}}}\Big) \sqrt{|t-s|}.
  \end{align*}
  \vspace{-\baselineskip}
  \end{part*}
  \begin{part*}
  It follows from \autoref{part:compactness} that the set 
  $\big\{m(t):m\in\bigcup_\lambda \FPm_\lambda\big\}$
  is pre-compact for a~fixed $t\in\Tb$.
  Then, in \autoref{part:equicontinuity},
  we showed that the family $\bigcup_\lambda\FPm_\lambda$ is equicontinuous in $\CPX$.
  Hence $\bigcup_{\lambda}\FPm_\lambda\subset\CPX$
  is pre-compact by the Arzel\`a--Ascoli theorem~\cite[\S7~Theorem~17]{MR0370454}.
  \vspace{-14pt}\[\ \]\vspace{-\belowdisplayskip}
  \end{part*}
 \end{proof}
 
 In the general case we are unable to prove uniqueness of solutions of \autoref{eq:fp}.
 However, we can make the following observation about the sets of solutions.
 \begin{corollary}\label{cor:m-compact}
  Assume~\ref{L:levy},~\ref{a1':b},~\ref{a:m}.
  If $\FPm\subset\CPX$ is the set of solutions of \autoref{eq:fp}
  corresponding to $(b,m_0)$, then $\FPm$ is convex, compact, and 
  \begin{align*}
      \sup_{m\in\FPm\vphantom{\Tb}}\sup_{t\in\Tb}m(t)[V]
      \leq c_1,
      \qquad\sup_{m\in\FPm\vphantom{T}}\sup_{0<|t-s|\leq T}\frac{\|m(t)-m(s)\|_0}{\sqrt{|t-s|}}
      \leq c_2,
   \end{align*}
  for a Lyapunov function $V$ such that $m_0[V],\|\fL V\|_\infty<\infty$
  (see \autoref{rem:levy-lyapunov}), and 
  \begin{align*}c_1= m_{0}[V] + T\|b\|_\infty \|\fL V\|_\infty,\qquad c_2
  = 2 + \big(2\sqrt{T}+K_d\big)\|b\|_\infty\|\fL\|_{\text{\textit{LK}}}.\end{align*}
 \end{corollary}
 \begin{proof}
  It follows from \autoref{def:fp-weak} that $\FPm$ is convex (the equation is linear),
  as well as that if $\{m_n\}\subset\FPm$ and $m_n\to \widehat m$ in $\CPX$,
  then $\widehat m\in\FPm$, i.e.~the set $\FPm$ is closed.
  Hence, by \autoref{lemma:fp-tightness}\autoref{item:compactness},
  we obtain that $\FPm\subset\CPX$ is compact.
  The specified bounds follow from \autoref{lemma:fp-tightness}\autoref{item:equicontinuity}
  and \eqref{eq:m-bound}.
 \end{proof}
 
 We now prove a kind of a stability result for solutions (in terms of 
 semicontinuity with respect to upper Kuratowski limits (see~\cite[\S29.III]{MR0217751}).
 \begin{lemma}\label{lemma:fp-continuity}
  Assume~\ref{L:levy},~\ref{a:m},
  and $\{b_n,b\}_{n\in\N}$ satisfy~\ref{a1':b} with a uniform bound by $B$.
  Let $\{\FPm_n,\FPm\}$ be the corresponding sets of solutions of \autoref{eq:fp}
  with $m_0$ as initial conditions.
  If $m_n\in\FPm_n$ for every $n\in\N$
  and $b_n(t)\to b(t)$ uniformly on compact sets in $\X$ for every $t\in\T$,
  then there exists a subsequence $\{m_{n_k}\}$ and $m\in\FPm$
  such that $m_{n_k}\to m$ in $\CPX$.
 \end{lemma}
 \begin{proof}
  By \autoref{lemma:fp-tightness}\autoref{item:compactness}
  the set $\bigcup_n\FPm_n\subset\CPX$ is pre-compact,
  and by \autoref{lemma:fp-tightness}\autoref{item:tightness}
  for every $\epsilon>0$ there exists a compact set $K_\epsilon\subset\X$ such that 
  \begin{align*}
   \sup_{n\in\N\vphantom{\Tb}}\sup_{m\in\FPm_n\vphantom{\Tb}}\sup_{t\in\Tb}
   m(t)(K_\epsilon^c)\leq \epsilon.
  \end{align*}
  Let $\{m_{n_k}\}\subset\{m_n\}$ be a convergent subsequence and $m=\lim_{k\to\infty}m_{n_k}$.
  Without loss of generality, we may still denote $m_{n_k}$ as $m_n$.
  For every $\varphi\in C_c^\infty(\X)$ we have
  \begin{align*}
      \bigg|\int_0^t (b_nm_n-b m)(\tau)[\fL\varphi]\,d\tau\bigg|
    = \bigg|\int_0^t\Big((b_n-b)m_n + b(m_n- m)\Big)(\tau)[\fL\varphi]\,d\tau\bigg|.
  \end{align*}
  Since $m_n\to m$ in $\CPX$ and $b\in\Cb$, we notice that
  \begin{align*}
   \limp_{n\to\infty\vphantom{\T}}\sup_{\tau\in\T}\big|m_n(\tau)[b(\tau)]-{m}(\tau)[b(\tau)]\big| = 0.
  \end{align*}
  Next,
  \begin{multline*}
   \bigg|\int_0^t (b_n-b)m_n(\tau)[\fL\varphi]\,d\tau \bigg|
    \leq \|\fL\varphi\|_\infty
     \int_0^T
     \int_{K_\epsilon \cup K_\epsilon^c}|b_n-b|(\tau,x)\,m_n(\tau,dx)\,d\tau \\
     \leq \|\fL\varphi\|_\infty\bigg(\epsilon T\big(\|b_n\|_{\infty}+\|b\|_\infty\big)+
     \int_0^T\sup_{x\in K_\epsilon}\big|b_n(\tau,x)-b(\tau,x)\big|\,d\tau\bigg).
  \end{multline*}
  We have $\big|b_n(t,x)-b(t,x)\big|\leq 2B$ for every $(t,x)\in\T\times\X$
  and $b_n(t)\to b(t)$ uniformly on compact sets in $\X$ for every $t\in\T$, hence
   $\sup_{x\in K_\epsilon}\big|b_n(t,x)-b(t,x)\big|\to 0$
   pointwise in $t\in\T$.
  Thus, by Lebesgue dominated convergence theorem,
   \begin{align*}
   \sup_{t\in\Tb}\limp_{n\to\infty\vphantom{\Tb}} 
   \bigg|\int_0^t (b_n m_n-b m)(\tau)[\fL\varphi]\,d\tau\bigg| 
   \leq 2\, \epsilon BT \|\fL\varphi\|_\infty.
   \end{align*}
   Since $\epsilon>0$ may be arbitrarily small and $m_n$ are solutions of \autoref{eq:fp},
   because of \autoref{lemma:defn-fp-sol}\autoref{item:D-small},
     $$(m(t) - m_0)[\varphi] = \lim_{n\to\infty} (m_n(t) - m_0)[\varphi]\\
   = \lim_{n\to\infty}\int_0^t b_n m_n(\tau)[\fL\varphi]\,d\tau = \int_0^t b m(\tau)[\fL\varphi]\,d\tau.$$
  Thus $m$ is a solution of \autoref{eq:fp} with parameters $b$ and $m_0$, i.e.~$ m\in\FPm$.
  \end{proof}
  \begin{remark}\label{rem:fp-stability}
   When the solutions of \autoref{eq:fp} are unique,
   \autoref{lemma:fp-continuity} is a standard stability result.
   Indeed, let $\{m_n,m\}$ be (the unique) solutions
   of \autoref{eq:fp} with a fixed initial condition~$m_0$
   and parameters $\{b_n, b\}$ such that $b_n\to b$ uniformly on compact sets in $\X$ for every $t\in\T$.
   By \autoref{lemma:fp-continuity} every subsequence of $\{m_n\}$ 
   has a further subsequence convergent to $m$.
   Thus $m_n\to m$ in $\CPX$.
  \end{remark}
 Next we show that the set of solutions is non-empty.
 \begin{theorem}\label{thm:fp-existence}
  Assume~\ref{L:levy},~\ref{a1':b},~\ref{a:m}.
  \hyperref[eq:fp]{Problem~\eqref{eq:fp}} has a very weak solution.
 \end{theorem}
 \begin{proof*}
 \begin{step*}{Approximate problem}
  For $\epsilon\in(0,1)$, let $\fL^\epsilon$ be the sequence of approximations
  of operator $\fL$ given by \autoref{lemma:approx}
  and $\nu^\epsilon$, $\fL^{\epsilon\,*}$ be their L\'evy measures and adjoint operators,
  respectively.
  
  By~\eqref{eq:fl-1} and the Fubini theorem, for every $\mu\in L^1(\X)$ we have
  \begin{align}\label{eq:fl-0}
   \int_\X\fL^{\epsilon\,*} \mu\,dx 
   = \int_\X\int_\X \big(\mu(x-z)-\mu(x)\big)\,dx\,\nu^\epsilon(dz)=0.
  \end{align}
 
  Let $b_\epsilon = b+\epsilon$ and $\mu_{0,\epsilon} = m_0*\rho_\epsilon$,
  where $\{\rho_\epsilon\}_{\epsilon\in(0,1)}$
  is the sequence of standard mollifiers.
  For $\epsilon\in(0,1)$ we consider the following family of problems 
  \begin{align}\label{eq:approx}
   \left\{
   \begin{aligned}
    &\dt \mu = \fL^{\epsilon\,*}(b_\epsilon \mu)\quad&\text{on $\T\times\X$},\\
    &\mu(0)=\mu_{0,\epsilon}\quad&\text{on $\X$}.
   \end{aligned}
   \right.
  \end{align}
   \end{step*}
 \begin{step*}{Existence of approximate solution $\mu_\varepsilon$\label{step:contraction}} 
  For $\mu\in\CLX$, define
  \begin{align}\label{eq:operator-G}
   \mathcal{G}_\epsilon(\mu)(t) 
   = \mu(0) + \int_0^t \fL^{\epsilon\,*}\big(b_\epsilon\mu\big)(\tau)\,d\tau.
  \end{align}
  We observe that for every $t_0\in\T$, because $\|b_\epsilon\|_\infty<\|b\|_\infty +1$,
  \begin{align*}
   \mathcal{G}_\epsilon : 
   C\big([0,t_0],L^1(\X)\big)\to C\big([0,t_0],L^1(\X)\big) \cap C^1\big((0,t_0],L^1(\X)\big)
  \end{align*}
  is a bounded linear operator.
 
  Let $\mu_1,\mu_2\in\CLX$ be such that $\mu_1(0)=\mu_2(0)$ and take
   $t_\epsilon = \frac{\epsilon^3}{4\,c_{\fL} \|b_\epsilon\|_\infty}$,
  where $c_{\fL}$ is the constant given by \autoref{lemma:approx}.
  Then, because of \autoref{lemma:approx}\autoref{item:eps-cube},
  \begin{multline*}
   \sup_{t\in[0,t_\epsilon]}\|\mathcal{G}_\epsilon(\mu_1-\mu_2)(t)\|_{L^1(\X)} 
   = \sup_{t\in[0,t_\epsilon]}
   \bigg\|\int_0^t\fL^{\epsilon\,*}\big(b_\epsilon(\mu_1-\mu_2)\big)(\tau)\,d\tau \bigg\|_{L^1(\X)}\\
   \leq t_\epsilon\frac{2\,c_{\fL} \|b_\epsilon\|_\infty }{\epsilon^3}
    \sup_{t\in[0,t_\epsilon]}\|\mu_1-\mu_2\|_{L^1(\X)}
   \leq \frac12\sup_{t\in[0,t_\epsilon]}\|\mu_1-\mu_2\|_{L^1(\X)}.
  \end{multline*}
  Therefore, by the Banach fixed point theorem, \autoref{eq:approx}
  has a unique solution $\mu_{\epsilon}\in C\big([0,t_\epsilon],L^1(\X)\big)$ for every $\epsilon>0$.
  Since $t_\epsilon>0$ is constant for fixed \mbox{$\epsilon>0$},
  we may immediately extend this solution to the interval $\Tb$
  and conclude that \autoref{eq:approx} has a unique solution in the space 
  $\CLX \cap C^1\big(\T,L^1(\X)\big)$.
 \end{step*}
 \begin{step*}{Compactness of $\{\mu_\varepsilon\}$ in $\CPX$}\label{step:compactness}\footnote{First we show that $\{\mu_\epsilon\}\subset C\big(\T,\mathcal{P}_{ac}(\X)\big)$, then we establish its tightness.}
  Because of the regularity of $\mu_\epsilon$ obtained in \autoref{step:contraction}, we have
  \begin{align}\label{eq:approx-L1-soln}
   \dt \mu_\epsilon = \fL^{\epsilon\,*} (b_\epsilon\mu_\epsilon)
   \quad\text{in}\quad
   C\big(\T,L^1(\X)\big).
  \end{align}
  Consider $\sgn(u)^-=\mathbbm{1}_{\{u<0\}}$.
  Then, by \eqref{eq:approx-L1-soln},
  \begin{align*}
   \int_0^t\int_\X\dt \mu_\epsilon \sgn (\mu_\epsilon)^-\,dx\,d\tau
   = \int_0^t\int_\X\fL^{\epsilon\,*} (b_\epsilon\mu_\epsilon)\sgn(\mu_\epsilon)^-\,dx\,d\tau.
  \end{align*}
 Since $b_\epsilon>0$, we have  $\sgn (\mu_\epsilon)^-= \sgn (b_\epsilon \mu_\epsilon)^-$ and for arbitrary real functions $u,v$,
 $v \sgn(u)^- \geq v\sgn(v)^-=-(v)^-$.
Therefore
  \begin{multline*}
   \big(\fL^{\epsilon\,*} (b_\epsilon\mu_{\epsilon})\sgn(\mu_\epsilon)^-\big)(x)
   = \int_\X\!\! \big(b_\epsilon\mu_{\epsilon}(x-z)-b_\epsilon\mu_{\epsilon}(x)\big)
     \sgn (b_\epsilon\mu_\epsilon)^-(x)\,\nu^\epsilon(dz)\\
    \geq -\int_\X \big((b_\epsilon\mu_{\epsilon})^-(x-z)
     -(b_\epsilon\mu_\epsilon)^-(x)\big)\,\nu^\epsilon(dz)
   =-\fL^{\epsilon\,*}\big((b_\epsilon\mu_{\epsilon})^-\big)(x).
  \end{multline*}
  By~\eqref{eq:fl-0},
  $\int_\X \fL^{\epsilon\,*}\big((b_\epsilon\mu_{\epsilon})^-
 \big)\,dx = 0$.
  Hence
  \begin{align*}
   0&\leq \int_0^t\int_\X \dt \mu_\epsilon\sgn (\mu_\epsilon)^-\,dx\,d\tau 
    = \int_0^t\int_\X-\dt (\mu_\epsilon)^-\,dx\,d\tau \\
   &= \int_\X(\mu_{0,\epsilon})^-\,dx - \int_\X(\mu_\epsilon)^-(t)\,dx.
  \end{align*}
  Since $\mu_{0,\epsilon}=m_0*\rho_\epsilon\geq 0$, i.e.~$(\mu_{0,\epsilon})^-=0$, and $(\mu_\epsilon)^-\geq 0$, this implies
  \begin{align*}
   0\leq \int_\X(\mu_\epsilon)^-(t)\,dx \leq  \int_\X(\mu_{0,\epsilon})^-\,dx = 0.
  \end{align*}
  Therefore $\mu_\epsilon(t)\geq 0$ for every $t\in\Tb$.
  
  By \autoref{step:contraction}, $\mu_{\epsilon}$ is the fixed point of $\mathcal{G}_\epsilon$.
  Thus, because of~\eqref{eq:fl-0},~\eqref{eq:operator-G}, and the Fubini--Tonelli theorem, we have
  \begin{align*}
   \int_\X\mu_\epsilon(t)\,dx 
   = \int_\X \mu_{0,\epsilon}\,dx + \int_0^t\int_\X \fL^{\epsilon\,*} (b_\epsilon\mu_{\epsilon})\,dx\,d\tau = 1.
  \end{align*}
  This, together with $\mu_\epsilon\geq 0$, means that $\mu_\epsilon(t)\in\mathcal{P}_{ac}(\X)$ for every
  $t\in\Tb$.
  Since $\mu_\epsilon\in\CLX$, it follows that $\mu_\epsilon\in C\big(\Tb,\mathcal{P}_{ac}(\X)\big)$.
 
  Notice that $\|b_\epsilon\|_\infty \leq \|b+1\|_\infty <B+1$ by~\ref{a1':b}.
  Let $V$ be a Lyapunov function such that $m_0[V],\|\fL V\|_\infty <\infty$
  (see \autoref{rem:levy-lyapunov}).
  By~\autoref{def:lyapunov},
  \begin{align*}
   \mu_{0,\epsilon}[V] =   (m_0*\rho_\epsilon)[V]
   \leq  m_0[V] + \|\nabla V\|_\infty \int_{B_1} |z|\,\rho_{\epsilon}(z)\,dz 
   \leq  m_0[V] +1.
  \end{align*}
  In combination with \autoref{lemma:approx}\autoref{item:fl-bound} we get
  \begin{align*}
   \sup_{\epsilon\in(0,1)}
   \big(b_\epsilon+\mu_{0,\epsilon}[V]+\|\fL^\epsilon V\|_\infty+\|\fL^\epsilon\|_{\text{\textit{LK}}}\big)
   <\infty.
  \end{align*}
  It follows from \autoref{lemma:fp-tightness} that the family $\{\mu_\epsilon\}$
  is pre-compact in $\CPX$.
 \end{step*}
 \begin{step*}{Passing to the limit}
  Using the result of~\autoref{step:compactness},
  let $\epsilon_k$ be a sequence such that 
  $\mu_{\epsilon_k}\to m$ in $\CPX$.
  By~\eqref{eq:approx-L1-soln},
  for every \mbox{$\epsilon_k$, $\varphi\in C_c^\infty(\X)$} and $s,t \in\Tb$, because $b_{\epsilon_k}-b = \epsilon_k$,
  \begin{align*}
  \begin{split}
   &\mu_{\epsilon_k}(t)[\varphi]-\mu_{\epsilon_k}(s)[\varphi] 
   = \int_s^t\int_\X\big(\fL^{\epsilon_k\,*} (\mu_{\epsilon_k} b_{\epsilon_k})\big)\varphi\,dx\,d\tau \\
   &= \epsilon_k\int_s^t\mu_{\epsilon_k}[\fL^{\epsilon_k}\varphi]\,d\tau
   +\int_s^t\mu_{\epsilon_k}\big[b(\fL^{\epsilon_k}\varphi-\fL\varphi)\big]\,d\tau
      +\int_s^t\mu_{\epsilon_k}[b\fL\varphi]\,d\tau.
  \end{split}
  \end{align*}
  Since $\lim\limits_{k\to\infty}\|\fL^{\epsilon_k}\varphi-\fL\varphi\|_\infty = 0$,
  by \autoref{lemma:approx}\autoref{item:fl-convergence}
  and the H\"older inequality,
  \begin{align*}
   m(t)[\varphi]-m(s)[\varphi] = \int_s^t m(\tau)[b(\tau)\fL\varphi]\,d\tau.
  \end{align*}
  It follows that $m$ is a very weak solution of \autoref{eq:fp}
  (see \autoref{lemma:defn-fp-sol}\autoref{item:D-small}).
 \vspace{-14pt}\[\ \]\vspace{-\belowdisplayskip}
 \end{step*}
 \end{proof*}
\subsection{Uniqueness}
Uniqueness for \autoref{eq:fp} holds when $b$ is more regular:
 \begin{description*}\vspace{0.33\baselineskip}
\item[(B$\,^{\prime}$)\label{a1'':b}] $b$ satisfies~\ref{a1':b}; in addition,
  $b\in\LC{\hb}$ for some $\hb>0$.
 \end{description*}
 \smallskip
 This condition is valid for $b=F'(\fL u)$ when $F'\in \Holder{\gamma}(\R)$ with $\gamma>0$ and $\fL u$ is smooth (\autoref{thm:hjb-interior}\autoref{item:uniform-cont-local}).
 \begin{theorem}\label{thm:fp-uniqueness}
  Assume~\ref{a1'':b} on $[0,t]$ for every $t\in\T$ and \ref{a:m}. If 
  either\smallskip
  \begin{enumerate}
      \item\label{item:fp-uniq-ndeg1} \ref{L:ndeg-loc},   $b\geq\kappa$ for some $\kappa>0$, and 
      $b\in\UC\big([0,t]\times\X\big)$ for every $t\in\T$; or
      \item\label{item:fp-uniq-ndeg} \ref{L:ndeg-nloc}
      and  $b\geq\kappa$ for some $\kappa>0$;
  \end{enumerate}\smallskip
  then \autoref{eq:fp}
  has precisely one very weak solution.
 \end{theorem} 
 
 We show uniqueness of solutions of \autoref{eq:fp}
 using a Holmgren-type argument.
 The idea is to use a solution of the ``dual'' equation,
\begin{align}\label{eq:fp-dual}
 \left\{
 \begin{aligned}
   \dt w &- b\fL w=0\quad&\text{on $\T\times\X$},\\
   w(0) &= \phi\quad&\text{on $\X$},
 \end{aligned}
 \right.
\end{align} 
as a test function in \autoref{def:fp-weak}.
For simplicity we consider a forward-in-time problem and then reverse time 
in the proof of \autoref{thm:fp-uniqueness}. 
We need sufficient regularity of solutions (see \autoref{lemma:defn-fp-sol})
when $\phi$ is taken from a dense subset of  $C(\X)$. 
Because of the non-degeneracy of the operator $\fL$
and the standard uniform ellipticity assumption $b\geq\kappa>0$,
existing results suffice to conclude.\footnote{In \cite{IJK03b}
we prove uniqueness for a degenerate case
of the Fokker--Planck equation~\eqref{eq:fp}.}
  \begin{lemma}\label{thm:ext_reg_soln_fp-dual}
Under the assumptions of \autoref{thm:fp-uniqueness} 
there exists a bounded classical solution of \autoref{eq:fp-dual}.
 \end{lemma}
 \begin{proof}
 \begin{part*}
  The statement follows from~\cite[Theorem~5.1.9]{MR3012216}
  (see \cite[page~175]{MR3012216} for relevant notation).
  \end{part*}
 \begin{part*} 
  Because $\phi \in C_c^\infty(\X)$, we have $\fL \phi \in C^\infty_b(\X)$ and thus
  by~\ref{a1'':b} we get $b\,\fL \phi\in\CC{\hb}$.
  Notice that $w$ is a bounded classical solution of \autoref{eq:fp-dual}
  if and only if $v=w-\phi$ is a bounded classical solution of
  \begin{align}\label{eq:fp-dual_zeroint}
   \left\{
   \begin{aligned}
    & \dt v - b\fL v=b\fL \phi\quad&\text{on $\T\times\X$},\\
    & v(0) = 0\quad&\text{on $\X$}.
   \end{aligned}
   \right.
  \end{align} 

  We study \autoref{eq:fp-dual_zeroint} using the results in~\cite{MR3201992}. We write $b\,\fL=A+B$, where
 \begin{align}
 \label{eq:operator_A}
  \hspace{-.5em}(A\phi)(t,x) = \int_\X \big(\phi(t,x+z)-\phi(t,x) 
  -\mathbbm{1}_{[1,2)}(2\sigma)\, z\cdot \nabla\phi(x) \big)
  b(t,x)\frac{\widetilde k(z)}{|z|^{d+2\sigma}}\,dz,
 \end{align}
 $\widetilde k(z) = \mathbbm{1}_{B_1}k(z) + \mathbbm{1}_{B_1^c}k(\frac{z}{|z|})$  is a normal extension of $k$ (defined in \ref{L:ndeg-nloc}) to $\X$, 
 and $B=b\,\fL-A:C_b(\X)\to C_b(\X)$ is a bounded operator (with L\'evy measure supported on $B_1^c$).
  We check the assumptions for operators $A$ and $B$ given by~\eqref{eq:operator_A}.
  Assumption \textbf{A} in~\cite{MR3201992} is satisfied,
  because we assume~\ref{L:ndeg-nloc},~\ref{a1'':b}, and $b\geq \kappa>0$.
  To verify assumptions \textbf{B1} and \textbf{B2} in~\cite{MR3201992},
  we choose $c(t, x, \upsilon) =\upsilon$, $U_n = B_1$, and $\pi = \nu|_{B_1^c}$
  (in the notation of~\cite{MR3201992}) and again use~\ref{L:ndeg-nloc},~\ref{a1'':b}.
  
  By~\cite[Theorem 4]{MR3201992} there exists a unique solution $v$ of \autoref{eq:fp-dual_zeroint}
  such that $\fL v\in\LC{\hb}$
  and $\dt v \in\Cb$ (see \cite[Definition~3]{MR3201992}).
  Thus $w = v- \phi$ is a bounded classical solution of \autoref{eq:fp-dual}.
  \vspace{-14pt}\[\ \]\vspace{-\belowdisplayskip}
  \end{part*}
 \end{proof} 
 
 \begin{proof}[Proof of \autoref{thm:fp-uniqueness}]
  Existence of a very weak solution follows by \autoref{thm:fp-existence}.
  Fix arbitrary $\varphi \in C_c^\infty(\X)$ and $t_0\in(0,T]$, and
  take $\widetilde{b}(t) = b(t_0-t)$ for every $t\in[0,t_0]$.
  Replace $b$ by $\widetilde{b}$ in \autoref{eq:fp-dual}.
  Then there exists a bounded classical solution $\widetilde{w}$ of \autoref{eq:fp-dual}
  --- by \autoref{thm:ext_reg_soln_fp-dual}.
Let $w(t)= \widetilde{w}(t_0-t)$ for $t\in [0,t_0]$.
  Then $\dt w, \fL w\in C\big((0,t_0)\times\X\big)$, and $w$ is a bounded classical solution of 
  \begin{align}\label{eq:fp-dual-back-nondeg}
   \left\{
   \begin{aligned}
    \dt w(t) &+ b(t)\fL w(t)=0\quad\text{in $(0,t_0)\times \X$},\\
    w(t_0) &= \varphi.
   \end{aligned}
   \right.
  \end{align}

  Suppose $m$ and $\widehat{m}$ are two very weak solutions of \autoref{eq:fp}
  with the same initial condition $m_0$ and coefficient $b$.
  By \autoref{def:fp-weak} (see \autoref{lemma:defn-fp-sol}\autoref{item:D-large})
  and~\eqref{eq:fp-dual-back-nondeg},
  \begin{align*}
   \big(m(t_0)-\widehat{m}(t_0)\big)[\varphi]
   =\int_0^{t_0} \big(m(\tau)-\widehat{m}(\tau)\big)\big[\dt w + b \fL(w)\big]\,d\tau=0.
  \end{align*}
  Hence, for every $t \in (0,T]$ and $\varphi\in C_c^\infty(\X)$,
   $(m(t) - \widehat{m}(t))[\varphi] =0$,
  which means that $m(t)=\widehat m(t)$ in $\PX$.
 \end{proof}
 \begin{corollary}\label{cor:b-uniq}
  Assume~\ref{a:F2}, \ref{a:m}, \ref{D:ndeg}.
  Condition \ref{U2} is satisfied if either
  \begin{align*}
      &(i)\quad \text{\ref{F:ndeg} and \ref{L:ndeg-loc}}
      & or &
      &(ii)\quad \text{\ref{a:F1'} and  \ref{L:ndeg-nloc}}
  \end{align*}
 \end{corollary}
 \begin{proof}
  Let $u_1,\,u_2\in \HJcD$ and $v_1=\fL u_1$, $v_2=\fL u_2$.
  Since $F'\in \Holder{\gamma}(\R)$ by \ref{a:F1}, we may consider 
  \begin{align*}
   b(t,x) = \int_0^1 F'\big(sv_1(t,x)+(1-s)v_2(t,x)\big)\,ds.
  \end{align*}
  Because $u_1,\,u_2\in \HJcD$ and $F'\geq0$, we have $b\in C\big(\T\times\X\big)$ and
  $b\geq0$.
  \smallskip\par
  \begin{part*}
  By \autoref{lemma:hjb-local-wang}
  and \autoref{thm:hjb-interior}\autoref{item:uniform-cont-local},
  $v_1,\,v_2 \in \Cb$ and 
  $v_1,\,v_2\in B\big([0,t],\Hb{\hd}\big)\cap\UC\big([0,t]\times\X\big)$
  for every $t\in\T$.
  Thus $b$ satisfies~\ref{a1'':b} on $[0,t]$ with $\hb = \gamma\hd$
  and $b\in\UC\big([0,t]\times\X\big)$.
  Since $F'\geq\kappa>0$, we have $b\geq \kappa>0$
  and \ref{U2} follows from \autoref{thm:fp-uniqueness}\autoref{item:fp-uniq-ndeg1}.
  \end{part*}
   \begin{part*}
  By \autoref{lemma:regularity}
  and \autoref{thm:hjb-interior}\autoref{item:uniform-cont-local},
  $v_1,\,v_2 \in \Cb$ and 
  $v_1,\,v_2\in B\big([0,t],\Hb{\hd}\big)$
  for every $t\in\T$.
  Thus $b$ satisfies~\ref{a1'':b} on $[0,t]$ with $\hb = \gamma\hd$.
  Since $F'\geq\kappa>0$, we have $b\geq \kappa>0$
  and \ref{U2} follows from \autoref{thm:fp-uniqueness}\autoref{item:fp-uniq-ndeg}. 
  \vspace{-14pt}\[\ \]\vspace{-\belowdisplayskip}
  \end{part*}
  \end{proof} 
 \end{section}
\begin{section}{The Mean Field Game system}\label{sec:mfg}
 In this section we prove existence and uniqueness for \autoref{eq:mfg} under general assumptions.
 These results yield a proof of \autoref{thm:main}.
 For the proof of existence, based on the Kakutani--Glicksberg--Fan fixed point theorem,
 we need to recall some ter\-mi\-no\-logy concerning set-valued maps.
 \begin{definition}
  A set-valued map $\mathcal{K}: X \to 2^Y$ is compact
  if the image $\mathcal{K}(X) = \bigcup\{\mathcal{K}(x) : x \in X\}$
  is contained in a compact subset of $Y$.
 \end{definition}
 \begin{definition}
  A set-valued map $\mathcal{K}:X\to 2^Y$ is upper-semicontinuous if,
  for each open set $A\subset Y$, the set  $\mathcal{K}^{-1}(2^A) =\{x:\mathcal{K}(x) \subset A\}$ is open.
  \end{definition}
 \begin{theorem}[Kakutani--Glicksberg--Fan~{\cite[\S7~Theorem~8.4]{MR1987179}}]\label{thm:kakutani}
  Let $\mathcal{S}$ be a convex subset of a normed space and
  $\mathcal{K} : \mathcal{S}\to 2^\mathcal{S}$ be a compact set-valued map.
  If $\mathcal{K}$ is upper-semicontinuous with non-empty compact convex values,
  then $\mathcal{K}$ has a fixed point, i.e.~there exists $x\in\mathcal{S}$ such that
  $x\in\mathcal{K}(x)$.\qed
 \end{theorem}
 In addition, the following lemma lets us express upper-semicontinuity in terms of sequences, which 
 are easier to handle (cf.~\autoref{lemma:fp-continuity}).
 \begin{lemma}[{\cite[\S43.II~Theorem~1]{MR0259835}}]\label{lemma:semi-cont}
  Let $X$ be a Hausdorff space and $Y$ a compact metric space.
  A set-valued compact map $\mathcal{K}: X \to 2^Y$ is upper-semicontinuous
  if and only if the conditions 
  \begin{align*}
   x_n\to x\ \text{in}\ X,&& y_n\to y\ \text{in}\ Y,&&\text{and}&& y_n\in\mathcal{K}(x_n)
  \end{align*}
   imply $y\in\mathcal{K}(x)$.\qed
  \end{lemma}
 \begin{theorem}\label{thm:mfg-existence} 
  Assume~\ref{L:levy},~\ref{a:F1},~\ref{a:m},~\ref{a:fg1},~\ref{R1}, \ref{R2}, \ref{R3}.
  Then there exists a classical--very weak solution of \autoref{eq:mfg}.
 \end{theorem}
 \begin{proof*}
  Let $X=\big(C\big(\Tb,\Mb\big),\sup_t\|\cdot\|_0\big)$ (see \autoref{def:rubinstein}).
  We want to find a solution of \autoref{eq:mfg} in $X$ by applying the Kakutani--Glicksberg--Fan
  fixed point theorem.
  To this end, we shall define a map $\mathcal{K}:\mathcal{S}\to2^\mathcal{S}$
  on a certain compact, convex set $\mathcal{S}\subset X$.
  Then the map $\mathcal{K}$ is automatically compact and we may use \autoref{lemma:semi-cont}
  to obtain upper-semicontinuity.
  \smallskip\par
  \begin{step}
  Let $V$ be a Lyapunov function such that $m_0[V],\|\fL V\|_\infty<\infty$ (see \autoref{rem:levy-lyapunov}).
  Define
  \begin{multline*}
   \mathcal{S} = \Big\{\mu \in\CPX : \mu(0)=m_0,\\
   \sup_{t\in\Tb}\mu(t)[V]\leq c_1,
   \quad \sup_{0<|t-s|\leq T}\frac{\|\mu(t)-\mu(s)\|_0}{\sqrt{|t-s|}}\leq c_2\Big\},
  \end{multline*}
  where $m_0$ is fixed and satisfies~\ref{a:m}, and 
  \begin{align*}
  \begin{split}
    c_1=  m_{0}[V] + T\KHJ \|\fL V\|_\infty,\qquad
    c_2=2 + \big(2\sqrt{T}+K_d\big)\KHJ\|\fL\|_{\text{\textit{LK}}}.
    \end{split}
  \end{align*}
  The set $\mathcal{S}$ is clearly convex.
  In addition, $\mathcal{S}$ is compact because of \autoref{prop:pre-compactness},
  the assumed equicontinuity in time, and the Arzel\`a--Ascoli theorem.
  \end{step}
  \begin{step} 
  Take $\mu\in\mathcal{S}$ and let 
   $f=\cF(\mu)$ and
   $\uT = \cG\big(\mu(T)\big)$.
  We define a map $\mathcal{K}_1:\mathcal{S}\to\Cbb$ by $\mathcal{K}_1(\mu)=u$,
  where $u$ is the unique bounded classical solution of \autoref{eq:hjb},
  corresponding to data $(f,\uT)$.
  The map $\mathcal{K}_1$ is well-defined because of~\ref{R1},~\ref{R3}, and \autoref{thm:hjb-viscosity}.
  By~\ref{a:F1} we find that $b=F'(\fL u)$ satisfies~\ref{a1':b}.
  
   We define a set-valued map $\mathcal{K}_2$ by $\mathcal{K}_2(u) = \FPm$,
   where $\FPm$ is the set of very weak solutions of \autoref{eq:fp} corresponding to $b=F'(\fL u)$.
   The set $\FPm\subset \mathcal{S}\subset\CPX$ is convex, compact,
   and non-empty because of \autoref{cor:m-compact} and \autoref{thm:fp-existence}.
    Now we define the fixed point map
   $\mathcal{K}(\mu) = \mathcal{K}_2(\mathcal{K}_1(\mu))= \FPm$.
  Because of its construction,
  $\mathcal{K}:\mathcal{S}\to2^\mathcal{S}$ is a compact map with non-empty compact convex values.
  \end{step}
  \begin{step}\label{step:u-cont}
  It remains to show that the map $\mathcal{K}:\mathcal{S}\to2^\mathcal{S}$ is
  upper-semicontinuous.
  Let $\{\mu_n,\mu\}_{n\in\N}\subset \mathcal{S}$ be such that $\lim\limits_{n\to\infty}\mu_n=\mu$
  and let $\{u_n,u\} = \{\mathcal{K}_1(\mu_n),\mathcal{K}_1(\mu)\}$
  be the corresponding solutions of \autoref{eq:hjb},
  and $\{\FPm_n,\FPm\} = \{\mathcal{K}(\mu_n),\mathcal{K}(\mu)\}$
  be the corresponding sets of solutions of \autoref{eq:fp}.
   
   Since $\lim\limits_{n\to\infty}\mu_n=\mu$, by~\ref{a:fg1}, \autoref{thm:hjb-viscosity}\autoref{item:cp}\footnote{This result can be obtained directly for classical solutions  (under~\ref{R3}) by a straightforward application of the maximum principle property of the L\'evy operator $\fL$.}, and~\ref{R2},
   we obtain $\fL u_n \to \fL u$ uniformly on compact sets in $\X$ for every $t\in\T$.
   Hence, if we let 
    $b_n = F'(\fL u_n)$
    and $b = F'(\fL u)$,
   then by~\ref{a:F1}, $b_n\to b$ uniformly on compact sets in $\X$ for every $t\in\T$.
   Moreover, the functions $b_n$ and $b$ satisfy~\ref{a1':b} and are uniformly bounded, by \ref{R3}.
   
   Consider a sequence $m_n\in\FPm_n$ and suppose it converges to some $\widehat m\in\mathcal{S}$.
   Then we use \autoref{lemma:fp-continuity} to say that $\widehat m\in\FPm$.
   This proves that the map $\mathcal{K}$ is upper-semicontinuous by \autoref{lemma:semi-cont}.
  \end{step}
 \begin{step}
  We now use \autoref{thm:kakutani} to get a fixed point $\widehat m\in\mathcal{S}$ of the map $\mathcal{K}$.
  Because of how $\mathcal{K}$ is defined, we have
   $\widehat m \in\mathcal{K}(\widehat m) = \mathcal{K}_2(\mathcal{K}_1(\widehat m))$.
  Thus there exists $\widehat u = \mathcal{K}_1(\widehat m)$,
  which is a bounded classical solution of \autoref{eq:hjb}
  with $f = \cF(\widehat m)$ and $\uT = \cG(\widehat m(T))$,
  and $\|F'(\fL \widehat u)\|_\infty\leq \KHJ$ by~\ref{R3}.
  Note that~$\widehat m$ is a very weak solution of
  \autoref{eq:fp} with $\widehat m(0)=m_0$ and $b = F'(\fL\widehat u$).
  This, in turn, means that the pair $(\widehat u,\widehat m)$
  is a classical--very weak solution of \autoref{eq:mfg}
  (see \autoref{def:mfg}).
 \vspace{-14pt}\[\ \]\vspace{-\belowdisplayskip}
 \end{step}
 \end{proof*}
 \begin{remark}
 Adding assumption \ref{U2} to \autoref{thm:mfg-existence}, 
 yields singleton-valued maps 
$\mathcal{K}_2:\HJcD\to2^\mathcal{S}$
 and $\mathcal{K}:\mathcal{S}\to2^\mathcal{S}$,
 and hence both are continuous (see~\autoref{step:u-cont}, \autoref{rem:fp-stability}).
 To conclude we may then use the classical Schauder theorem~\cite[\S6~Theorem~3.2]{MR1987179} (a special case of the Kakutani--Glicksberg--Fan theorem, cf.~\autoref{lemma:semi-cont}).
 \end{remark}
 \begin{theorem}\label{thm:mfg-uniqueness}
  Assume~\ref{L:levy},~\ref{a:F1},~\ref{a:F2},~\ref{a:m},~\ref{a:fg2},~\ref{U1}, \ref{U2}.
  Then \autoref{eq:mfg} has at most one solution.
 \end{theorem}
 \begin{proof}
  Suppose $(u_1,m_1)$ and $(u_2,m_2)$ are classical--very weak solutions of \autoref{eq:mfg}
  (see \autoref{def:mfg}), and take 
  $u=u_1-u_2$, and $m=m_1-m_2$.
  To shorten the notation further, let $\fL u_1 = v_1$, $\fL u_2 = v_2$, and $v=v_1-v_2$.
  
  By \autoref{def:mfg}, $u_1,\,u_2$ are bounded classical solutions of \autoref{eq:hjb},
  and by \ref{U1}, $\{\dt u_1,\,\dt u_2,\,\fL u_1,\,\fL u_2\}\subset\Cb$.
  By \ref{a:F1}, $F'(v_1),\,F'(v_2)\in\Cb$,  thus $u\in\mathcal{U}$,
  where $\mathcal{U}$ is defined in \autoref{lemma:defn-fp-sol}\autoref{item:D-large}.
  Further, $m_1,\,m_2$ are very weak solutions of \autoref{eq:fp} and satisfy~\eqref{eq:fp-weaksolution}
  for every $\phi\in\mathcal{U}$  by \autoref{lemma:defn-fp-sol}\autoref{item:D-large}.
  Hence,
  \begin{align}\label{eq:duality}
  \begin{split}
   m&(T)\big[u(T)\big]-m(0)\big[u(0)\big]\\
   &=\big(m_1(T)-m_2(T)\big)\big[u_1(T)-u_2(T)\big]-\big(m_1(0)-m_2(0)\big)\big[u_1(0)-u_2(0)\big]\\
    &= \int_0^T \Big(m_1\big[\dt u + F'(v_1) v\big] - m_2\big[\dt u + F'(v_2) v\big]\Big)(\tau)\,d\tau.
  \end{split}
  \end{align}
  As $m_1(0)=m_2(0) = m_0$, we have 
   $m(0)\big[u(0)\big]=0$
   and, thanks to~\ref{a:fg2}, 
   \begin{align*}
   m(T)[u(T)] = \big(m_1(T)-m_2(T)\big)\big[\cG\big(m_1(T)\big)-\cG\big(m_2(T)\big)\big]\leq 0.
  \end{align*}  
Hence by \eqref{eq:duality} we get 
\begin{align}\label{eq:dual_leseqaul}
    \int_0^T \Big(m_1\big[\dt u + F'(v_1) v\big] - m_2\big[\dt u + F'(v_2) v\big]\Big)(\tau)\,d\tau\leq0.
\end{align}
We further notice that 
   $\dt u + F(v_1)-F(v_2)=\cF(m_2)-\cF(m_1)$.
  Then, by integrating this expression with respect to the measure $m$, we obtain
  \begin{align}\label{eq:equalitym1m2}
      \int_0^T m\,[\dt u + F(v_1)-F(v_2)](\tau) \, d\tau
      = \int_0^T m\, [\cF(m_2)-\cF(m_1)] \, d\tau.
  \end{align}
  From~\ref{a:F2} we know that $F$ is convex, thus
  \begin{align}\label{eq:m1m2ineq}
   F(v_1) - F(v_2) \leq  F'(v_1)\,v\quad\text{and}\quad F(v_1) - F(v_2) \geq  F'(v_2)\,v, 
  \end{align}  
  and since $m_1, m_2 \in\CPX$ are non-negative measures,
  by \eqref{eq:equalitym1m2}, \eqref{eq:m1m2ineq} and~\ref{a:fg2},
  \vspace{-\baselineskip}
  \begin{align}\label{eq:dual_leseqaul2}
   \begin{split}
    \int_0^T m_1\big[\dt u + F'(v_1)\,v\big](\tau)\,d\tau
    -\int_0^T m_2\big[\dt u + F'(v_2)\,v\big](\tau)\,d\tau& \\
    \geq \int_0^T(m_1-m_2)[\cF(m_2)-\cF(m_1)](\tau)\,d\tau& \geq 0.
   \end{split}
  \end{align}
Combining \eqref{eq:dual_leseqaul} and \eqref{eq:dual_leseqaul2}, we find that 
\begin{align*}
    \int_0^T m_1\big[\dt u + F'(v_1)\,v\big](\tau)\,d\tau
      -\int_0^T m_2\big[\dt u + F'(v_2)\,v\big](\tau)\,d\tau =0.
\end{align*}
  Then, taking into account \eqref{eq:equalitym1m2}, we get
  \begin{multline*}
   0 = \int_0^T\int_\X \big(F'(v_1)\,v-F(v_1) + F(v_2)\big)\,m_1(\tau,dx)\,d\tau \\
   + \int_0^T\int_\X \big(F(v_1) - F(v_2)-F'(v_2)\,v\big)\,m_2(\tau,dx)\,d\tau.
  \end{multline*}
  By \eqref{eq:m1m2ineq},
  both functions under the integrals are non-negative and continuous,
  thus in particular
 \begin{align}\label{eq:supp}
    F(v_1) - F(v_2)-F'(v_1)(v_1-v_2)=0\quad\text{on $\supp m_1$},
  \end{align}
  where by $\supp m_1$ we understand the support of $m_1$
  taken as a measure on $\Tb\times\X$.

  Let $(t,x)\in\supp m_1$. If $v_1(t,x)\neq v_2(t,x)$, then by \eqref{eq:supp}
  \begin{align*}
  F'(v_1(t,x)) = \frac{F(v_1(t,x)) - F(v_2(t,x))}{v_1(t,x)-v_2(t,x)}.
  \end{align*}
  This means that the tangent line to the graph of $F$ at $v_1(t,x)$
  and the secant line joining $F(v_1(t,x))$ and $F(v_2(t,x))$ coincide.
  By \ref{a:F1} and \ref{a:F2} both lines also coincide
  with the tangent at $v_2(t,x)$, thus $F'(v_1(t,x))=F'(v_2(t,x))$.
  Of course if $v_1(t,x)=v_2(t,x)$, then $F'(v_1(t,x))=F'(v_2(t,x))$ as well.
  Therefore $m_1$ can be written as a solution of \autoref{eq:fp}
  with $F'(v_2)$ in place of $F'(v_1)$.
  By \ref{U2} we get $m_1=m_2$.
  Then also $u_1 = u_2$ by \autoref{thm:hjb-viscosity}.
 \end{proof}
\end{section}
\appendix
\section{Proofs of some technical results}\label{appendix:comparison}
 \begin{proof}[Outline of proof of \autoref{lemma:regularity}]
  We employ the method of continuity, following the scheme of the proof of \cite[Theorem 13.9.1]{MR3837125}, using the Schauder estimates of \cite[Theorem 1.3]{MR3803717}, and the existence results for linear problems in \cite[Theorem 4]{MR3201992}.
  
  For $s\in[0,1]$ consider the family of problems
  \begin{align*}\tag{$P$\hspace{-2pt}\raisebox{-2pt}{\scriptsize{$s$}}}
    \left\{\begin{aligned}
    -\dt u &= (1-s)\fL u + s F(\fL u)+f,\qquad f\in \HHb{\hd/2\sigma}{\hd}\\
    u(T)&=g.
    \end{aligned}\right.
  \end{align*}
  with regularized initial data $g$.
 Let 
 \begin{multline*}
     S=\{s: P_s\text{ has a classical solution satisfying}\\ \text{the interior $(\tfrac{\alpha}{2\sigma},\alpha)$-regularity estimates}\}\subset[0,1].
 \end{multline*}
  We have $0\in S$ by \cite[Theorem 4]{MR3201992}\footnote{\cite{MR3201992} shows
  well-posedness of a strong solution with the correct spatial regularity uniform in time. Given our assumptions, time regularity can then be obtained from the equation in the standard way.} and $S$ is closed by \cite[Theorem 1.3]{MR3803717}. Next we show that $S$ is open. Let $u_0$ be a solution of problem $P_{s_0}$. Consider the map
 \begin{align*}
 \Psi^s:\HHb{\sigma+\alpha/2\sigma}{2\sigma+\alpha}\to \HHb{\sigma+\alpha/2\sigma}{2\sigma+\alpha}
 \end{align*} 
 given by 
 $\Psi^s(w)=v$, where $v$ is a solution to the linear problem
 \begin{align*}
     \left\{\begin{aligned}
      -\dt v &= (1-s)\fL v +s F'(\fL u_0)\fL v + s\big(F(\fL w)-F'(\fL u_0)\fL w\big) +f,\\
      v(T)&= g.
     \end{aligned}\right.
 \end{align*}
  It is well-defined because of \cite[Theorem 4]{MR3201992}.
  We use a second-order approximation as in \cite[Theorem 13.9.1]{MR3837125},
  and again \cite[Theorem 4]{MR3201992},
  to get that $\Psi^s$ is a self-map on a certain neighbourhood of $u_{0}$. 
  Then we show that $\Psi^s$ is a contraction on this set. 
  The fixed point given by the Banach theorem is a solution of problem $P_s$ for $0<|s_0-s|<\epsilon$. Note that the computations are essentially the same as in \cite{MR3837125} because the problems depend linearly 
  on the time derivatives. 
  
 The case of general data $g$ follows by an approximation argument. 
 \end{proof}
\begin{proof}[Proof of \autoref{prop:lyapunov}]
  We proceed in steps, constructing successive functions
  that accumulate properties required by \autoref{def:lyapunov}
  and are adequately integrable.
  \smallskip\par
  \begin{step*}{Integrability, monotonicity, unboundedness}
  The conclusion of this step is essentially stated in \cite[Example~8.6.5\,(\emph{ii})]{MR2267655}, but a complete proof is lacking and the precise function $v_0$, which we need, cannot be extracted.
 Let 
  \begin{align*}
   v(x) = v_0\big(|x|\big),\qquad\text{where}\qquad v_0(t) = \sup_{m\in\Pi}m\{x: |x|\geq t\}.
  \end{align*}  
  Then \mbox{$v_0:[0,\infty)\to[0,1]$} is non-increasing and $v_0(0) = 1$.
  Because $\Pi$ is tight, we have \mbox{$\lim\limits_{t\to\infty}v_0(t) = 0$}.
  Thus, $-\log(v_0):[0,\infty) \to [0,\infty]$ is non-decreasing,
  $\log(v_0(0))=0$, and \mbox{$\lim\limits_{t\to\infty}-\log(v_0(t)) = \infty$}.
  For $m\in\Pi$, let $\Phi^{m}(\tau) = m\circ v^{-1}\big([0,\tau)\big)$.
  Then,\footnote{Notice that 
  $\big\{x: m\{y:|y|\geq|x|\} <\tau \big\}=\{x:|x|>r_\tau\}$,
  while 
  $\big\{x: m\{y:|y|>|x|\} \leq \tau \big\}=\{x:|x|\geq r_\tau\}$,
  where $r_\tau$ is such that
  $m\{x:|x|> r_\tau\}\leq \tau\leq m\{x:|x|\geq r_\tau\}$.
  If $m$ is absolutely continuous with respect to the Lebesgue measure,
  then the measure $m$ of both sets is equal to $\tau$.
  Choosing the correct inequality in the definition of the function $v_0$
  is essential.}
  \begin{align*}
      \Phi^m(\tau) = m\big(v^{-1}\big([0,\tau)\big)\big) 
      &= m\big\{x: \forall\, \widehat m\in\Pi\quad \widehat m\{y:|y|\geq|x|\} <\tau \big\}\\
      &\leq m\big\{x: m\{y:|y|\geq|x|\} <\tau \big\}\leq \tau.
  \end{align*}
  Integrating \emph{by substitution}~\cite[Theorem~3.6.1]{MR2267655}
  and \emph{by parts}~\cite[Exercise~5.8.112]{MR2267655},\footnote{
  From~\cite[Exercise~5.8.112\,(\emph{i})]{MR2267655} we get 
  $\int_r^1 -\log(\tau)\,d\Phi^m(\tau) = \int_r^1 \frac{\Phi^{m}(\tau)}{\tau}\,d\tau$ for every $r>0$.
  Then we may pass to the limit $r\to0$ by the monotone convergence theorem,
  cf.~\cite[Exercise~5.8.112\,(\emph{iii})]{MR2267655}.}
  \begin{align}\label{eq:log-v}
      \int_\X -\log\big(v(x)\big)\,m(dx) = \int_0^1 -\log(\tau)\,d\Phi^m(\tau)
      = \int_0^1 \frac{\Phi^m(\tau)}{\tau}\,d\tau \leq \int_0^1\, d\tau .
  \end{align}
  \end{step*}
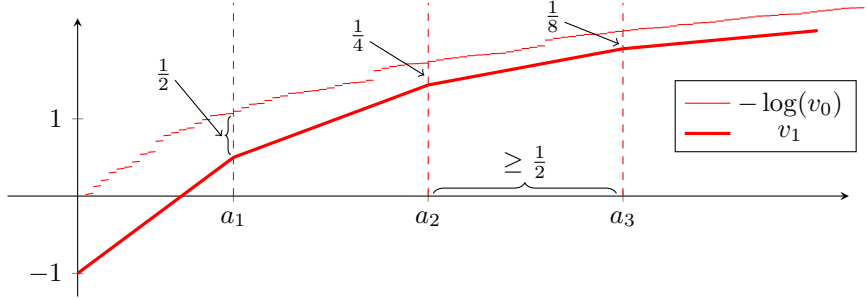
\begin{figure}[t]
 \begin{tikzpicture}
    \pgfmathsetmacro{\yup}{0}; 
    \begin{axis}[axis lines = middle,enlargelimits,clip=false,width=\textwidth,height=15em,ymin=-4,ymax=8,
    ytick={-4,0,4},xtick={0,4,9,14},domain=0:18, xticklabels={$0$,$a_1$,$a_2$,$a_3$},yticklabels={$-1$,$0$,$1$},
    xmax=18,legend style={at={(axis cs:20,6)}}]
    \addplot[domain=0:4,samples=3,red] {3/2*x-4};\addlegendentry{$-\log(v_0)$}
    \addplot[domain=0:4,samples=3,red, very thick] {3/2*x-4};\addlegendentry{$v_1$}
    \foreach \k in {1,...,9} {
     \addplot[domain=\k*0.2:\k*0.2+0.2,color=red,samples=5] {2.6*ln(-3*ln(1/(x/20+1))+\yup+1)};
     \pgfmathparse{\yup+0.3*rnd)}
     \xdef\yup{\pgfmathresult}
    };
    \foreach \k in {10,...,37} {
     \addplot[domain=\k*0.2:\k*0.2+0.2,color=red,samples=5] {2.95*ln(-3*ln(1/(x/20+1))+\yup+1)};
     \pgfmathparse{\yup+0.3*rnd)}
     \xdef\yup{\pgfmathresult}
    };
    \foreach \k in {38,...,59} {
     \addplot[domain=\k*0.2:\k*0.2+0.2,color=red,samples=5] {3.19*ln(-3*ln(1/(x/20+1))+\yup+1)};
     \pgfmathparse{\yup+0.3*rnd)}
     \xdef\yup{\pgfmathresult}
    };
    \foreach \k in {60,...,100} {
     \addplot[domain=\k*0.2:\k*0.2+0.2,color=red,samples=5] {3.3*ln(-3*ln(1/(x/20+1))+\yup+1)};
     \pgfmathparse{\yup+0.3*rnd)}
     \xdef\yup{\pgfmathresult}
    };
    \addplot[domain=4:9,samples=3,red,very thick] {3/4*x-1};
    \addplot[domain=9:14,samples=3,red,very thick] {3/8*x+2.375};
    \addplot[domain=14:19,samples=3,red,very thick] {3/16*x+5};
    \addplot[dashed,samples=3,red] coordinates {(4,0)(4,10)};
    \addplot[dashed,samples=3,red] coordinates {(9,0)(9,10)};
    \addplot[dashed,samples=3,red] coordinates {(14,0)(14,10)};
    \addplot [arrows=->] coordinates{(2.5,6) (3.7,3.1)};
    \addplot [arrows=->] coordinates{(7.5,8) (8.9,6.25)};
    \addplot [arrows=->] coordinates{(12.5,9) (13.9,7.875)};
    \addplot [] (2.2,6.3) node {$\frac12$};
    \addplot [] (7.2,8.3) node {$\frac14$};
    \addplot [] (12.2,9.3) node {$\frac18$};
    \addplot [] (11.5,1.75) node {$\geq\frac12$};
    \addplot [decorate,decoration={brace,amplitude=2pt},xshift=-1pt,yshift=1pt]
coordinates{(4,2) (4,4)};
    \addplot [decorate,decoration={brace,amplitude=4pt},xshift=-1pt,yshift=1pt]
coordinates{(9.2,0) (13.9,0)};
  \end{axis}
 \end{tikzpicture}
 \caption{Comparison of $-\log(v_0)$ and $v_1$.}
 \label{fig:v0v1}
 \end{figure}
  \begin{step*}{Continuity, concavity}\footnote{Concavity serves as an intermediate step
  to obtain subadditivity.}
  For $N\in\N\cup\{\infty\}$ and sequences $\{a_n\}$, $\{b_n\}$ to be fixed later,
  let $v_1:[0,\infty)\to [-1,\infty)$ be the piecewise affine function given by (see~\autoref{fig:v0v1})
  \begin{align*}
      v_1(t) = \sum_{n=0}^N\, l_n(t)\mathbbm{1}_{[a_n,a_{n+1})}(t),
      \quad\text{where}\quad l_n(t) = {2^{-n}}(t-a_n)+b_n.
  \end{align*}
  We set $a_0= 0$.
  For $n\in\N$, when $a_n<\infty$, let $b_{n} = -\log\big(v_0(a_{n})\big)-2^{-n}$ and
  \begin{align*}
      a_{n+1} = \inf A_n,\quad\text{where}\quad A_n=\Big\{t\geq a_n:-\log\big(v_0(t)\big)-l_n(t)
      \leq {2^{-n-1}}\Big\}.
  \end{align*}
  We put $\inf\emptyset = \infty$ and $N = \sup\{n : a_n<\infty\}$.
  Note that for every $n<N+1$,
  \begin{align*}
   -\log\big(v_0(a_{n})\big) -v_1(a_{n}) =-\log\big(v_0(a_{n})\big) -b_n =2^{-n}   
  \end{align*}
  and on the interval $[a_n,a_{n+1}]$,
  \begin{align*}
   -\log(v_0)-v_1 \geq 2^{-{n-1}}\qquad
   \text{(hence $-\log\big(v_0(t)\big)\geq v_1(t)$ for every $t\geq 0$)}.
  \end{align*}
  To verify continuity, take a sequence $\{s_k\}\subset A_n$ such that $\lim\limits_{k\to\infty}s_k= a_{n+1}$.
  Then, because $-\log(v_0)$ is non-decreasing
  and $l_n$ is continuous,
  \begin{align*}
      -\log\big(v_0(a_{n+1})\big) - l_{n}(a_{n+1})
      \leq \liminf_{k\to\infty} \Big(-\log\big(v_0(s_k)\big) - l_{n}(s_k)\Big) \leq 2^{-n-1}.
  \end{align*}
  Thus
   $-\log\big(v_0(a_{n+1})\big) - l_{n}(a_{n+1})= 2^{-n-1}$,
  i.e.~$l_{n+1}(a_{n+1}) = b_{n+1} = l_n(a_{n+1})$, which implies that $v_1$ is continuous.
  Moreover, $a_{n+1}-a_n\geq\frac12$, since
  this distance is the shortest when $\log(v_0)$ is constant on $[a_n,a_{n+1}]$.
  We have $v_1(0)=-1$, $\lim\limits_{t\to\infty}v_1(t) = \infty$,
   and 
   \begin{align*}
   v_1' = \sum_{n=0}^N\, 2^{-n}\mathbbm{1}_{[a_n,a_{n+1})}\qquad
   \text{(a non-increasing function, see \autoref{fig:v1v2})},
   \end{align*} 
   which implies that $v_1$ is concave.
   In addition, $v_1(t)\leq t-1$, hence $v_1(1)\leq 0$.
  \end{step*}
    \begin{figure}[t!]
  \begin{tikzpicture}
    \begin{axis}[axis lines = middle,enlargelimits,clip=false,width=\textwidth,
    height=15em,ymin=0,ymax=2,ytick={0,1/2,1,2},xtick={0,3,4,5,8,9,10,13,14,15},
    domain=0:18, xticklabels={$0$,,$a_1$,,,$a_2$,,,$a_3$},yticklabels={$0$,$\frac14$,$\frac12$,$1$}]
    \addplot[domain=0:4,samples=3,red,dashed] {2};
    \addplot[domain=0:3,samples=3,red] {2};
    \addplot[domain=3:5,samples=100,red] {(1/4)*((x-4)^3-3*(x-4)+6)};
    \addplot[domain=5:8,samples=3,red] {1};
    \addplot[domain=8:10,samples=100,red] {(1/8)*((x-9)^3-3*(x-9)+6)};
    \addplot[domain=10:13,samples=3,red] {1/2};
    \addplot[domain=13:15,samples=100,red] {(1/16)*((x-14)^3-3*(x-14)+6)};
    \addplot[domain=15:18,samples=3,red] {1/4};
    \addplot[domain=4:9,samples=3,red,dashed] {1};
    \addplot[domain=9:14,samples=3,red,dashed] {1/2};
    \addplot[domain=14:18,samples=3,red,dashed] {1/4};
    \addplot[dashed,samples=3,red] coordinates {(4,1)(4,2)};
    \addplot[dashed,samples=3,red] coordinates {(9,1/2)(9,1)};
    \addplot[dashed,samples=3,red] coordinates {(14,1/4)(14,1/2)};
    \addplot [] (4,0.25) node {$\pm\frac18$};
    \addplot [decorate,decoration={brace,amplitude=4pt},yshift=1pt]
coordinates{(3.1,0) (4.9,0)};
    \legend{$v_1'$,$v_2'$}
  \end{axis}
 \end{tikzpicture}
 \caption{Comparison of $v_1'$ and $v_2'$}
 \label{fig:v1v2}
 \end{figure}
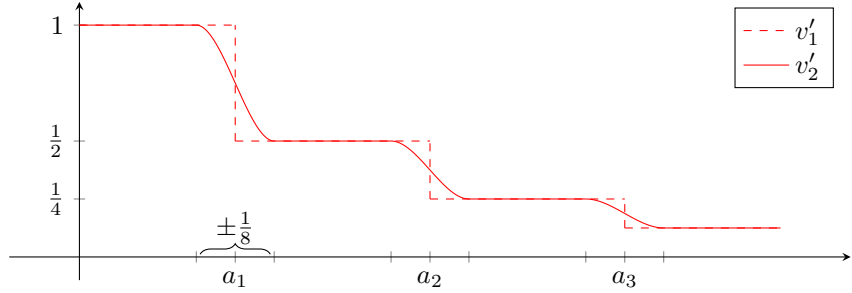
  \begin{step*}{Differentiability}
  Let $p(t)= \frac14(t^3-3t+6)\mathbbm{1}_{[-1,1)}(t)$.
  Then $p$ acts as a smooth transition between values $2$ and $1$
  on the interval $[-1,1]$, with vanishing derivatives at the end points.
  Let $v_2$ be such that $v_2(0)=-1$ and
  (see~\autoref{fig:v1v2})
  \begin{align*}
   v_2'(t) = \mathbbm{1}_{[0,a_{1}-\frac18)}(t)
   +\sum_{n=1}^N 2^{-n}\bigg(p\big(8(t-a_{n})\big)+\mathbbm{1}_{[a_{n}+\frac18,a_{n+1}-\frac18)}(t)\bigg).
  \end{align*}
   Then $v_2\in C^2\big([0,\infty)\big)$, $v_2$ is concave, increasing,
   and $\lim\limits_{t\to\infty}v_2(t)=\infty$.
   Moreover, 
   \begin{align*}\|v_2''\| \leq \sup_{t}\bigg|\frac12\,\frac{d}{dt} p(8t)\bigg| \leq 3.\end{align*}
   Next, we verify that $v_2\leq v_1$.
   Notice that for every $t\in[-1,1]$,
   \begin{align*}
    \int_{-1}^t p(s)\,ds 
    \leq \int_{-1}^t 2\cdot\mathbbm{1}_{[-1,0]}(s) + \mathbbm{1}_{[0,1]}(s)\,ds,
    \quad\text{and}\quad\int_{-1}^1 p(s)\,ds =3.
   \end{align*}
    By suitable scaling and shifting,
    for every $t\in\bigcup\limits_{n=1}^N \big[a_n-\frac18,a_{n}+\frac18\big]$
    we get $v_2(t) \leq v_1(t)$, and $v_2(t)= v_1(t)$ otherwise.
  \end{step*}
  \begin{step*}{Subadditivity, bounds on derivatives}
  Let $V_0 = \frac13 (v_2+1)$.
  Then $V_0:[0,\infty)\to [0,\infty)$ is concave and hence subadditive.
  Moreover, $V_0$ is increasing, $\lim\limits_{t\to\infty}V_0(t)=\infty$,
  and $\|V_0'\|_\infty, \|V_0''\|_\infty \leq 1$.
  This proves that $V(x) = V_0\big(\sqrt{1+|x|^2}\big)$ is a Lyapunov function.
  By subadditivity and monotonicity,
  \begin{align*}
   V_0\big(\sqrt{1+t^2}\big)\leq V_0(t+1)\leq V_0(t)+V_0(1),
  \end{align*}
  hence for every $m\in\Pi$, because $v_2\leq v_1\leq -\log(v_0)$ and by \eqref{eq:log-v},
  \begin{align*}
    0&\leq \int_\X V(x)\,m(dx) \leq V_0(1) + \int_\X V_0\big(|x|\big)\,m(dx) \\
    &\leq  \frac{v_2(1)+1}3+\frac13 - \frac13\int_\X \log(v(x))\,m(dx) \leq 
     \frac{v_1(1)}{3}  +\frac{1}{3}  +\frac{1}{3}  +\frac{1}{3} \leq 1 .
  \end{align*}
  This shows that $V$ is a Lyapunov function such that $m[V]\leq 1$ for every $m\in\Pi$.
\vspace{-14pt}\[\ \]\vspace{-\belowdisplayskip}
\end{step*}
\end{proof}
 \section{The Legendre--Fenchel transform}\label{appendix:legendre}
 For a comprehensive treatment of the Legendre--Fenchel transform we refer to \cite{MR0274683,MR1865628}.
 Below we gather the particular properties of cost functions $L$ needed to derive the model in \autoref{section:stochastics}, and corresponding to Hamiltonians $F$ satisfying \ref{a:F1} and \ref{a:F2}.
 These properties are expected, but in the setting we consider, we could not find the proofs in the literature.
 \begin{proposition}\label{prop:G-F}
 Let $L:[0,\infty)\to\R\cup\{\infty\}$ be a lower-semicontinuous function
 such that $L\not\equiv\infty$
 and define $F(z)= \sup_{\zeta\in[0,\infty)}(z \zeta-L(\zeta))$.\footnote{Taking the supremum over $[0,\infty)$ is consistent with extending $L$ by $L(\zeta)=\infty$ for $\zeta<0$ and taking the supremum over all of $\R$ as is usual.
 Conversely, if $\lim_{z\to-\infty}F(z)\neq\infty$, taking $L(\zeta) = \sup_{z\in \R}\big(\zeta z - F(z) \big)$ results in $L(\zeta)=\infty$ for $\zeta<0$. This is logical since $\zeta$ stands for the time rate (see \autoref{section:stochastics}) and the cost of going back in time should be prohibitive.}
 Then $F$ is convex and non-decreasing.
 In addition,\smallskip
 \begin{enumerate}
 \item\label{item:growth} if $\lim\limits_{\zeta\to\infty}L(\zeta)/\zeta=\infty$,
 then $F$ is finite-valued and locally Lipschitz-continuous;
 \item\label{item:conv} if $L$ is convex and is strictly convex on $\{L\neq\infty\}$,
 then $F$ is differentiable on $\{F\neq\infty\}$
 and $\zeta\mapsto z\zeta-L(\zeta)$ achieves its supremum at $\zeta=F'(z)$;
 \item\label{item:sconv} let $L$ be convex, $\lim\limits_{\zeta\to\infty}L(\zeta)/\zeta=\infty$ and $\partial L$ be the subdifferential of~$L$.
 If for every $\zeta_1,\zeta_2\in [0,\infty)$ and $z_1\in \partial L(\zeta_1)$ there exists $c_{z_1}>0$ such that for every $z_2\in \partial L(\zeta_2)$ satisfying $|z_1-z_2|\leq 1$ we have
 \begin{align*}\hspace{-.5\mathindent}
 (z_1-z_2)(\zeta_1-\zeta_2) \geq c_{z_1} |\zeta_1-\zeta_2|^{1+\frac{1}{\gamma}},\text{\footnotemark}    
 \end{align*}
 \footnotetext{When $\gamma=1$ and $c_{z_1}$ is in fact independent of $z_1$, this corresponds to the usual strong convexity of $L$ (see~\cite[Theorem~D.6.1.2]{MR1865628}); if $\gamma_1<\gamma_2$, the condition with $\gamma_1$ allows for a \emph{flatter} (less  non-affine) function $L$ than the one with $\gamma_2$.}
  then \mbox{$F'\in \Holder{\gamma}(\R)$}.
 \end{enumerate}
 \end{proposition}
 \begin{proof}
  The function $F$ is convex as a supremum of convex (affine) functions.
  For $\zeta,h\geq 0$ and $z\in\R$ we have $(z+h)\zeta-L(\zeta) \geq z \zeta-L(\zeta)$ and thus 
  \begin{align*}
      F(z+h)=\sup_{\zeta\in[0,\infty)}\big((z+h)\zeta-L(\zeta)\big)
      \geq \sup_{\zeta\in[0,\infty)}\big(z \zeta-L(\zeta)\big) = F(z).
  \end{align*}
  \begin{part*}
  Because $\lim\limits_{\zeta\to\infty}L(\zeta)/\zeta=\infty$,
  for every $z\in\R$,
      $\lim\limits_{\zeta\to\infty}\big(z-{L(\zeta)}/{\zeta}\big)\zeta = -\infty$.
  Since $L$ is lower-semicontinuous and $L\not\equiv\infty$,
  there exists $\zeta_0<\infty$ such that
  \begin{align*}
      L(\zeta_0)<\infty\qquad\text{and}\qquad
      \sup_{\zeta\in[0,\infty)}\Big(\big(z-\tfrac{L(\zeta)}{\zeta}\big)\zeta\Big) = z \zeta_0-L(\zeta_0).
  \end{align*} 
  As a convex function with finite values, $F$ is then locally Lipschitz-continuous.
  \end{part*}
  \begin{part*} Since $L$ is lower-semicontinuous,
  the statement follows from \cite[Theorem~23.5, Corollary~23.5.1, Theorem~26.3, page~52]{MR0274683}.
  \end{part*}
  \begin{part*}
   Note that \!\autoref{item:sconv} implies \!\autoref{item:growth} and \!\autoref{item:conv}
   (cf.~\cite[Theorem~D.6.1.2]{MR1865628})
   hence $F$ has finite values on $\R$ and $F'$ exists everywhere.
   If $z_i\in \partial L(\zeta_i)$, then $\zeta_i = F'(z_i)$ by \cite[Theorem~23.5]{MR0274683}.
   For $|z_1-z_2|\leq 1$ we thus have
   \begin{align*}
     |z_1-z_2||F'(z_1)-F'(z_2)| \geq c_{z_1} |F'(z_1)-F'(z_2)|^{1+\frac{1}{\gamma}}.
   \end{align*}
   which gives us $F'\in \Holder{\gamma}(\R)$ (see \eqref{eq:holder} in \autoref{def:holder}).
   \vspace{-14pt}\[\ \]\vspace{-\belowdisplayskip}
  \end{part*}
 \end{proof}
 \section*{Acknowledgements}
IC  was supported by the INSPIRE faculty fellowship (IFA22-MA187). ERJ received funding from the Research Council of Norway under Grant Agreement No. 325114 “IMod. Partial differential equations, statistics and data: An interdisciplinary approach to data-based modelling”.
 MK was supported by the Polish NCN grant 2016/23/B/ST1/00434 and Croatian Science Foundation grant IP-2018-01-2449.  
 The main part of the research behind this paper was conducted when  IC and MK were
 fellows of the ERCIM Alain Bensoussan Programme at NTNU.

 \bibliographystyle{siam}
 \let\OLDthebibliography\thebibliography
 \renewcommand\thebibliography[1]{
  \OLDthebibliography{#1}
  \setlength{\itemsep}{0pt plus 0.2ex}
 }
 \bibliography{MFG}

\begin{thebibliography}{10}

\bibitem{MR4214773}
{\sc Y.~Achdou, P.~Cardaliaguet, F.~Delarue, A.~Porretta, and F.~Santambrogio},
  {\em Mean field games}, CIME Lecture Notes in Mathematics, Springer, 2020.

\bibitem{MR4361908}
{\sc P.~D.~S. Andrade and E.~A. Pimentel}, {\em Stationary fully nonlinear
  mean-field games}, J. Anal. Math., 145 (2021), pp.~335--356.

\bibitem{MR0172689}
{\sc R.~J. Aumann}, {\em Markets with a continuum of traders}, Econometrica, 32
  (1964), pp.~39--50.

\bibitem{MR2484103}
{\sc O.~E. Barndorff-Nielsen and J.~Schmiegel}, {\em Time change, volatility,
  and turbulence}, in Mathematical control theory and finance, Springer, 2008,
  pp.~29--53.

\bibitem{MR3363697}
{\sc O.~E. Barndorff-Nielsen and A.~Shiryaev}, {\em Change of time and change
  of measure}, World Scientific Publishing, 2nd~ed., 2015.

\bibitem{BT22}
{\sc A.~Barrasso and N.~Touzi}, {\em Controlled diffusion mean field games with
  common noise and {M}c{K}ean-{V}lasov second order backward {SDE}s}, Theory
  Probab. Appl., 66 (2022), pp.~613--639.
\newblock Translation of Teor. Veroyatn. Primen. {{\bf{6}}6} (2021), 774--805.

\bibitem{MR3980873}
{\sc C.~Benazzoli, L.~Campi, and L.~Di~Persio}, {\em {$\varepsilon$}-{N}ash
  equilibrium in stochastic differential games with mean-field interaction and
  controlled jumps}, Statist. Probab. Lett., 154 (2019), pp.~108522, 8.

\bibitem{MR4158808}
\leavevmode\vrule height 2pt depth -1.6pt width 23pt, {\em Mean field games
  with controlled jump-diffusion dynamics: existence results and an illiquid
  interbank market model}, Stochastic Process. Appl., 130 (2020),
  pp.~6927--6964.

\bibitem{MR3134900}
{\sc A.~Bensoussan, J.~Frehse, and P.~Yam}, {\em Mean field games and mean
  field type control theory}, Springer, 2013.

\bibitem{MR2267655}
{\sc V.~I. Bogachev}, {\em Measure theory. {V}ol. {I}, {II}}, Springer, 2007.

\bibitem{MR3391701}
{\sc V.~I. Bogachev, M.~R\"{o}ckner, and S.~V. Shaposhnikov}, {\em Uniqueness
  problems for degenerate {F}okker--{P}lanck--{K}olmogorov equations}, J. Math.
  Sci. (N.Y.), 207 (2015), pp.~147--165.

\bibitem{MR2178045}
{\sc V.~S. Borkar}, {\em Controlled diffusion processes}, Probab. Surv., 2
  (2005), pp.~213--244.

\bibitem{MR3156646}
{\sc B.~B\"{o}ttcher, R.~Schilling, and J.~Wang}, {\em L\'{e}vy
  matters.~{III}}, Springer, 2013.

\bibitem{MR3992040}
{\sc F.~Camilli and R.~De~Maio}, {\em A time-fractional mean field game}, Adv.
  Differential Equations, 24 (2019), pp.~531--554.

\bibitem{MR3967062}
{\sc P.~Cardaliaguet, F.~Delarue, J.-M. Lasry, and P.-L. Lions}, {\em The
  master equation and the convergence problem in mean field games}, Princeton
  University Press, 2019.

\bibitem{MR3399179}
{\sc P.~Cardaliaguet, P.~J. Graber, A.~Porretta, and D.~Tonon}, {\em Second
  order mean field games with degenerate diffusion and local coupling}, NoDEA
  Nonlinear Differential Equations Appl., 22 (2015), pp.~1287--1317.

\bibitem{MR3752669}
{\sc R.~Carmona and F.~Delarue}, {\em Probabilistic theory of mean field games
  with applications.~{I}}, Springer, 2018.

\bibitem{MR3753660}
\leavevmode\vrule height 2pt depth -1.6pt width 23pt, {\em Probabilistic theory
  of mean field games with applications.~{II}}, Springer, 2018.

\bibitem{CARR2004113}
{\sc P.~Carr and L.~Wu}, {\em Time-changed {L}évy processes and option
  pricing}, Journal of Financial Economics, 71 (2004), pp.~113 -- 141.

\bibitem{MR3912635}
{\sc A.~Cesaroni, M.~Cirant, S.~Dipierro, M.~Novaga, and E.~Valdinoci}, {\em On
  stationary fractional mean field games}, J. Math. Pures Appl. (9), 122
  (2019), pp.~1--22.

\bibitem{MR3148110}
{\sc H.~Chang-Lara and G.~D\'{a}vila}, {\em Regularity for solutions of non
  local parabolic equations}, Calc. Var. Partial Differential Equations, 49
  (2014), pp.~139--172.

\bibitem{MR3115838}
\leavevmode\vrule height 2pt depth -1.6pt width 23pt, {\em Regularity for
  solutions of nonlocal parabolic equations.~{II}}, J. Differential Equations,
  256 (2014), pp.~130--156.

\bibitem{MR3592657}
{\sc E.~Chasseigne and E.~R. Jakobsen}, {\em On nonlocal quasilinear equations
  and their local limits}, J. Differential Equations, 262 (2017),
  pp.~3759--3804.

\bibitem{chowdhury2021numerical}
{\sc I.~Chowdhury, O.~Ersland, and E.~R. Jakobsen}, {\em On numerical
  approximations of fractional and nonlocal mean field games}, Found. Comput.
  Math., 23 (2023), pp.~1381--1431.

\bibitem{IJK03b}
{\sc I.~Chowdhury, E.~R. Jakobsen, and M.~Krupski}, {\em A strongly degenerate
  fully nonlinear mean field game with nonlocal diffusion}, in preparation,
  (2024).

\bibitem{MR3934106}
{\sc M.~Cirant and A.~Goffi}, {\em On the existence and uniqueness of solutions
  to time-dependent fractional {MFG}}, SIAM J. Math. Anal., 51 (2019),
  pp.~913--954.

\bibitem{MR4702626}
{\sc J.~C. Correa and E.~A. Pimentel}, {\em A {H}essian-dependent functional
  with free boundaries and applications to mean-field games}, J. Geom. Anal.,
  34 (2024), pp.~Paper No. 95, 21.

\bibitem{MR1118699}
{\sc M.~G. Crandall, H.~Ishii, and P.-L. Lions}, {\em User's guide to viscosity
  solutions of second order partial differential equations}, Bull. Amer. Math.
  Soc. (N.S.), 27 (1992), pp.~1--67.

\bibitem{djete2023mean}
{\sc M.~F. Djete}, {\em {Mean field games of controls: On the convergence of
  Nash equilibria}}, Ann. Appl. Probab., 33 (2023), pp.~2824 -- 2862.

\bibitem{MR3803717}
{\sc H.~Dong, T.~Jin, and H.~Zhang}, {\em Dini and {S}chauder estimates for
  nonlocal fully nonlinear parabolic equations with drifts}, Anal. PDE, 11
  (2018), pp.~1487--1534.

\bibitem{MR4309434}
{\sc O.~Ersland and E.~R. Jakobsen}, {\em On fractional and nonlocal parabolic
  mean field games in the whole space}, J. Differential Equations, 301 (2021),
  pp.~428--470.

\bibitem{MR2179357}
{\sc W.~H. Fleming and H.~M. Soner}, {\em Controlled {M}arkov processes and
  viscosity solutions}, Springer, 2nd~ed., 2006.

\bibitem{MR2058260}
{\sc I.~I. Gikhman and A.~V. Skorokhod}, {\em The theory of stochastic
  processes. {II}}, Springer, 2004.
\newblock Reprint of the 1975 edition.

\bibitem{MR3363751}
{\sc D.~A. Gomes, L.~Nurbekyan, and E.~A. Pimentel}, {\em Economic models and
  mean-field games theory}, Instituto Nacional de Matem\'{a}tica Pura e
  Aplicada (IMPA), Rio de Janeiro, 2015.

\bibitem{MR3559742}
{\sc D.~A. Gomes, E.~A. Pimentel, and V.~Voskanyan}, {\em Regularity theory for
  mean-field game systems}, Springer, 2016.

\bibitem{MR1987179}
{\sc A.~Granas and J.~Dugundji}, {\em Fixed point theory}, Springer, 2003.

\bibitem{MR2762362}
{\sc O.~Gu\'{e}ant, J.-M. Lasry, and P.-L. Lions}, {\em Mean field games and
  applications}, in Paris--{P}rinceton {L}ectures on {M}athematical {F}inance
  2010, Springer, 2011, pp.~205--266.

\bibitem{MR2380957}
{\sc F.~B. Hanson}, {\em Applied stochastic processes and control for
  jump-diffusions: Modeling, analysis, and computation}, Society for Industrial
  and Applied Mathematics (SIAM), 2007.

\bibitem{MR1865628}
{\sc J.-B. Hiriart-Urruty and C.~Lemar\'{e}chal}, {\em Fundamentals of convex
  analysis}, Springer, 2001.

\bibitem{MR2344101}
{\sc M.~Huang, P.~E. Caines, and R.~P. Malham\'{e}}, {\em An invariance
  principle in large population stochastic dynamic games}, J. Syst. Sci.
  Complex., 20 (2007), pp.~162--172.

\bibitem{MR2346927}
{\sc M.~Huang, R.~P. Malham\'{e}, and P.~E. Caines}, {\em Large population
  stochastic dynamic games: closed-loop {M}c{K}ean--{V}lasov systems and the
  {N}ash certainty equivalence principle}, Commun. Inf. Syst., 6 (2006),
  pp.~221--251.

\bibitem{MR1011252}
{\sc N.~Ikeda and S.~Watanabe}, {\em Stochastic differential equations and
  diffusion processes}, North-Holland Publishing Co; Kodansha, 2nd~ed., 1989.

\bibitem{MR2129093}
{\sc E.~R. Jakobsen and K.~H. Karlsen}, {\em Continuous dependence estimates
  for viscosity solutions of integro-{PDE}s}, J. Differential Equations, 212
  (2005), pp.~278--318.

\bibitem{jakobsen2023master}
{\sc E.~R. Jakobsen and A.~Rutkowski}, {\em The master equation for mean field
  game systems with fractional and nonlocal diffusions}, arXiv:2305.18867,
  (2023).

\bibitem{MR4223351}
{\sc P.~Jameson~Graber, V.~Ignazio, and A.~Neufeld}, {\em Nonlocal {B}ertrand
  and {C}ournot mean field games with general nonlinear demand schedule}, J.
  Math. Pures Appl. (9), 148 (2021), pp.~150--198.

\bibitem{MR1321597}
{\sc A.~S. Kechris}, {\em Classical descriptive set theory}, Springer, 1995.

\bibitem{MR0370454}
{\sc J.~L. Kelley}, {\em General topology}, Springer, 1975.
\newblock Reprint of the 1955 edition [Van Nostrand].

\bibitem{MR3667677}
{\sc Y.-C. Kim and K.-A. Lee}, {\em The {E}vans--{K}rylov theorem for nonlocal
  parabolic fully nonlinear equations}, Nonlinear Anal., 160 (2017),
  pp.~79--107.

\bibitem{MR661144}
{\sc N.~V. Krylov}, {\em Boundedly inhomogeneous elliptic and parabolic
  equations}, Izv. Akad. Nauk SSSR Ser. Mat., 46 (1982), pp.~487--523, 670.

\bibitem{MR901759}
\leavevmode\vrule height 2pt depth -1.6pt width 23pt, {\em Nonlinear elliptic
  and parabolic equations of the second order}, D. Reidel Publishing Co., 1987.

\bibitem{MR1406091}
\leavevmode\vrule height 2pt depth -1.6pt width 23pt, {\em Lectures on elliptic
  and parabolic equations in {H}\"{o}lder spaces}, American Mathematical
  Society, 1996.

\bibitem{MR3837125}
\leavevmode\vrule height 2pt depth -1.6pt width 23pt, {\em Sobolev and
  viscosity solutions for fully nonlinear elliptic and parabolic equations},
  American Mathematical Society, 2018.

\bibitem{MR0217751}
{\sc K.~Kuratowski}, {\em Topology. {V}ol. {I}}, Academic Press; Polskie
  Wydawnictwo Naukowe, 1966.

\bibitem{MR0259835}
\leavevmode\vrule height 2pt depth -1.6pt width 23pt, {\em Topology. {V}ol.
  {II}}, Academic Press; Polskie Wydawnictwo Naukowe, 1968.

\bibitem{MR3332857}
{\sc D.~Lacker}, {\em Mean field games via controlled martingale problems:
  existence of {M}arkovian equilibria}, Stochastic Process. Appl., 125 (2015),
  pp.~2856--2894.

\bibitem{MR0241822}
{\sc O.~A. Lady\v{z}enskaja, V.~A. Solonnikov, and N.~N. Ural'ceva}, {\em
  Linear and quasilinear equations of parabolic type}, American Mathematical
  Society, 1968.

\bibitem{MR2269875}
{\sc J.-M. Lasry and P.-L. Lions}, {\em Jeux \`a champ moyen. {I}. {L}e cas
  stationnaire}, C. R. Math. Acad. Sci. Paris, 343 (2006), pp.~619--625.

\bibitem{MR2271747}
\leavevmode\vrule height 2pt depth -1.6pt width 23pt, {\em Jeux \`a champ
  moyen. {II}. {H}orizon fini et contr\^{o}le optimal}, C. R. Math. Acad. Sci.
  Paris, 343 (2006), pp.~679--684.

\bibitem{MR2295621}
\leavevmode\vrule height 2pt depth -1.6pt width 23pt, {\em Mean field games},
  Jpn. J. Math., 2 (2007), pp.~229--260.

\bibitem{MR1465184}
{\sc G.~M. Lieberman}, {\em Second order parabolic differential equations},
  World Scientific Publishing, 1996.

\bibitem{MR3012216}
{\sc A.~Lunardi}, {\em Analytic semigroups and optimal regularity in parabolic
  problems}, Birkh\"{a}user, 2013.
\newblock Reprint of the 1995 original.

\bibitem{MR1432798}
{\sc R.~Mikulevi\v{c}ius and H.~Pragarauskas}, {\em Nonlinear potentials of the
  {C}auchy--{D}irichlet problem for the {B}ellman integro-differential
  equation}, Liet. Mat. Rink., 36 (1996), pp.~178--218.

\bibitem{MR3201992}
\leavevmode\vrule height 2pt depth -1.6pt width 23pt, {\em On the {C}auchy
  problem for integro-differential operators in {H}\"{o}lder classes and the
  uniqueness of the martingale problem}, Potential Anal., 40 (2014),
  pp.~539--563.

\bibitem{MR455113}
{\sc I.~Monroe}, {\em Processes that can be embedded in {B}rownian motion},
  Ann. Probability, 6 (1978), pp.~42--56.

\bibitem{MR3951822}
{\sc C.~Mou}, {\em Remarks on {S}chauder estimates and existence of classical
  solutions for a class of uniformly parabolic {H}amilton--{J}acobi--{B}ellman
  integro-{PDE}s}, J. Dynam. Differential Equations, 31 (2019), pp.~719--743.

\bibitem{MR2322248}
{\sc B.~\O~ksendal and A.~Sulem}, {\em Applied stochastic control of jump
  diffusions}, Universitext, Springer, Berlin, second~ed., 2007.

\bibitem{MR1046331}
{\sc B.~\O{}ksendal}, {\em When is a stochastic integral a time change of a
  diffusion?}, J. Theoret. Probab., 3 (1990), pp.~207--226.

\bibitem{MR2533355}
{\sc H.~Pham}, {\em Continuous-time stochastic control and optimization with
  financial applications}, Springer, 2009.

\bibitem{Ricciardi}
{\sc M.~Ricciardi}, {\em Mean {F}ield {G}ames {PDE} with {C}ontrolled
  {D}iffusion}, in Some {A}dvances in {M}ean {F}ield {G}ames {T}heory,
  Rome--Paris, 2020, pp.~133--173.
\newblock PhD Thesis.

\bibitem{MR0274683}
{\sc R.~T. Rockafellar}, {\em Convex analysis}, Princeton University Press,
  1970.

\bibitem{MR2607035}
{\sc M.~R\"{o}ckner and X.~Zhang}, {\em Weak uniqueness of {F}okker--{P}lanck
  equations with degenerate and bounded coefficients}, C. R. Math. Acad. Sci.
  Paris, 348 (2010), pp.~435--438.

\bibitem{MR3185174}
{\sc K.~Sato}, {\em L\'{e}vy processes and infinitely divisible distributions},
  Cambridge University Press, 2013.

\bibitem{MR1139064}
{\sc L.~Wang}, {\em On the regularity theory of fully nonlinear parabolic
  equations.~{II}}, Comm. Pure Appl. Math., 45 (1992), pp.~141--178.

\end{thebibliography}
\end{document}